\documentclass[12pt]{article}
\usepackage{stmaryrd}
\usepackage{amscd}
\usepackage{pifont}
\usepackage{amssymb}
\usepackage{amsmath}
\usepackage{amsthm}
\usepackage{enumerate}
\usepackage{graphicx}

\usepackage{ulem}
\usepackage{color}
\usepackage{xcolor}

\usepackage{chngcntr}

\newtheorem{theorem}{Theorem}[section] 
\newtheorem{lemma}[theorem]{Lemma}

\allowdisplaybreaks

\date{}

\usepackage[nospace]{cite}
\usepackage[footnotesep=0.65cm]{geometry}
\usepackage{geometry}
\usepackage{type1cm}
\usepackage[hang,marginal]{footmisc}
\let\oldfootnote\footnote
\renewcommand\footnote[1]{\oldfootnote{\renewcommand\baselinestretch{1.1}
\large\footnotesize\ignorespaces#1}}
\newcommand{\bn}{\mathbf{n}}

\newcommand{\bx}{\mathbf{x}}
\newcommand{\bd}{\mathbf{d}}

\numberwithin{equation}{section}

\usepackage[noblocks]{authblk}

\begin{document}

\title{\large{\textbf{An interior penalty discontinuous Galerkin method for the time-domain acoustic-elastic wave interaction problem}}}

\author[1]{\normalsize Yingda Cheng}
\author[2]{\normalsize Jing Huang}
\author[3]{\normalsize Xiaozhou Li}
\author[4]{\normalsize Liwei Xu}

\affil[1]
{\normalsize\itshape Department of Mathematics, Department of  Computational Mathematics, Science and Engineering, Michigan State University, East Lansing, MI 48824, U.S.A. ycheng@msu.edu. Research is supported by NSF grants  DMS-1453661 and DMS-1720023.
}
\affil[2]
{\normalsize\itshape Department of Mathematics,   Michigan State University, East Lansing, MI 48824, U.S.A.
}
\affil[3]
{\normalsize\itshape School of Mathematical Sciences, University of Electronic Science and Technology of China,  Sichuan 611731, China.  Research is supported by a NSFC grant (No. 11801062).
}
 \affil[4]
{\normalsize\itshape School of Mathematical Sciences, University of Electronic Science and Technology of China,  Sichuan 611731, China.  Research is supported by a Key Project of the Major Research Plan of NSFC (No.91630205) and a NSFC grant (No. 11771068).
}

\maketitle

\begin{abstract}
In this paper, we consider numerical solutions of a time domain acoustic-elastic wave interaction problem which occurs between a bounded penetrable elastic body and a compressible inviscid fluid. It is also called the fluid-solid interaction problem. First, we introduce an artificial boundary to transform the original transmission problem into a problem in a bounded region, and then we employ a symmetric interior penalty discontinuous Galerkin method for the solution of the reduced interaction problem consisting of second-order wave equations. A priori error estimate in the energy norm is presented, and numerical results confirm the accuracy and validity of the proposed method.

{\bf Keywords:} Fluid-solid interaction problem; Time-domain wave propagation; Second wave equation; Interior penalty discontinuous Galerkin methods; Error analysis

\end{abstract}

\section{Introduction}
\paragraph{}In a time-domain fluid-solid interaction (FSI) problem, an incident acoustic wave is scattered by a bounded elastic obstacle immersed in a homogeneous, compressible and inviscid fluid. The problem of determining the scattered wave field plays prominent roles in many scientific and engineering areas, such as detecting and identifying submerged objects, geophysical exploration, medical imaging, seismology, and oceanography (such as~\cite{GCH1994}). Because of the difficulties of dealing with the time dependence, the FSI problem is usually studied based on a time-harmonic setting. Various efficient and accurate numerical methods have been developed to solve the time-harmonic FSI problem.  Some of them provide proper guidance also for time-domain problems.  Popular methods include the boundary integral equation (BIE) method, see~\cite{CJLPAM1995, TYGCHLX2015} and coupling methods such as the so-called coupled FEM-BEM method, see~\cite{TYARLX}.  An artificial boundary or an absorbing layer (\cite{YX}) can be introduced to reduce the original unbounded problem to a bounded problem which could be solved using field equation solvers such as the finite element method.

   \par{}The studies of this time-domain scattering problems now gain more and more attention since the time-domain model gives more information about the wave, more general material, and nonlinearity, see~\cite{QCPM2014,LLWBWXZ2012}. There are relatively fewer mathematical analysis and numerical studies for the time-domain wave scattering problems.  As to the numerical studies towards the time-domain wave scattering problems, the main challenge is how to handle the problem defined on the unbounded domain.  Many approaches attempted to solve the time-domain problems numerically are developed, such as coupling of boundary element and finite element with different time quadratures, see~\cite{OVEHA1991,BFMKBIW2006},   time-domain boundary elemment methods ~\cite{GNS2017, GNS2018}, and to name a few.    The authors of~\cite{MJGJBK1995} gave the exact non-reflective boundary condition for the three dimensional time-domain wave scattering problem in 1995.  For basic isotropy wave equation with constant coefficients, a planar PML method in one space direction designed for some particular domain is considered in~\cite{TH1999}.  In~\cite{ZC2009}, the mathematical analysis of a time-domain DtN operator and the convergence analysis of the PML method for acoustic wave scattering were given.  In~\cite{LLWBWXZ2012}, the time-domain exact nonreflective boundary conditions in both two and three dimensions were computed and analyzed.  For the time-domain wave FSI problem, the rigorous mathematical study is still an open challenge, and related work include~\cite{XFPLYW2001,GCHFJSRJW2017,GLZ2017,BGL2018}. Concerning its numerical solutions, the authors of ~\cite{GCHSVTFJS2017} applied time-domain boundary integral equations to time-domain FSI problem and analyzed the resulting nonlocal initial-boundary problem, motivated by the time-harmonic FSI problems. Coupling methods are also utilized in ~\cite{OVEHA1991,XF2000,XFZX2004,CGMM2017,CGM2017} for solving the time-domain FSI problems.

   Instead of numerical methods mentioned above, in this paper, we focus on using discontinuous Galerkin (DG) methods, which has natural advantages for dealing with the time-domain FSI problem, see~\cite{DGMP2016}.  The original DG methods were proposed for the numerical solution of hyperbolic neutron transport equations, as well as for discontinuities in some elliptic and parabolic problems.  For second order wave equations, various DG methods have been proposed, which first reformulate the original wave problem to a first-order hyperbolic system, such as the local discontinuous Galerkin (LDG) methods in~\cite{BCCWS1989,XCS2013,Chou:2014}, and we refer to~\cite{Chung:2006,Chung:2009,Kaser:2008,WSBG2010} for a review of some other DG methods for the first order wave equations.
   The first DG method for the original second-order formulation of acoustic wave equation was proposed in~\cite{BRMFW2003}, which is based on a nonsymmetric formulation.  Here, we propose and analyze a symmetric interior penalty DG (SIPDG) method to solve the time-domain FSI problem.  It is the same method used in~\cite{MJGASDS2006} for the spatial discretization of the second-order scalar wave equation.  Compared to the nonsymmetric formulation in~\cite{BRMFW2003}, the symmetric discretization of the second-order form wave equation offers extra benefits such as a positive definite stiffness matrix and hence is free of any (unnecessary) damping.  One can refer to~\cite{Arnold:2002} for details of the DG methods for second order equations.  Finally we note that another DG formulation for wave equations in second order for is the energy based method proposed in~\cite{AH2018} and extended to the coupled acoustic-elastic problem in~\cite{AW2018}.

    \par{}The remainder of the paper is organized as follows. We first describe the original time-domain FSI problem in Section 2. Then in Section 3, the unbounded problem is reduced to a  bounded initial-boundary value problem.  In Section 4, we establish \textit{a priori} error estimates for the IPDG solution of the reduced problem. Numerical experiments are presented in Section 5 to confirm the theoretical results, and finally, a conclusion is made in Section 6.  For the sake of completeness, we provide the mathematical analysis towards well-posedness of the reduced problem introduced in Section 3 in the appendix.

\section{Model}
Here we study the same model as in~\cite{GCHFJSRJW2017}, where the authors gave mathematical analysis from the aspect of integral equation method. The statement of the model is as follows. Suppose $\Omega \subset \mathbb{R}^{2}$ is a bounded domain of a elastic body with boundary $\Gamma=\partial\Omega$, which is enclosed by the unbounded homogenous compressible inviscid fluid domain $\Omega^{+}=\mathbb{R}^{2} \backslash \overline{\Omega}$ (see  Figure~\ref{fig:fsi}), and a finite time interval $J = (0, T)$. Given an incident wave $\varphi^i$, the scattered wave $\varphi$ is generated by the induced solid $\Omega$ in the fluid domain $\Omega^{+}$.
In the fluid domain $\Omega^+$, the governing equations are the linearized Euler equation and the linearized equation of continuity
      \begin{equation}\label{eq:governed_equation}
      \rho_1\frac{\partial \mathbf{v}}{\partial t}+\nabla p = 0,\quad \frac{\partial p}{\partial t}+c^2\rho_1\text{div}~\mathbf{v} = 0,\quad(\mathbf{x},t)\in\Omega^{+}\times J,
      \end{equation}
      where $\mathbf{v}$ is the velocity field, $p$, $\rho_1$ and $c$ are the pressure, the density and the speed of sound in the fluid respectively.
   \begin{figure}
     \centering
     \includegraphics[height=5.0cm,width=8.5cm]{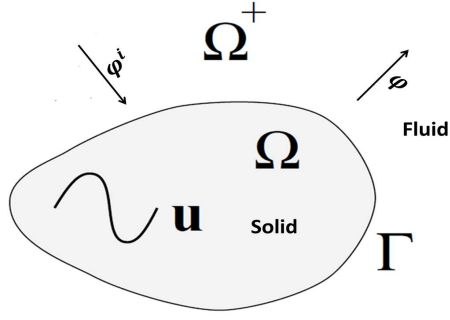}
     \caption{Illustration of fluid-solid interaction.}\label{fig:fsi}
   \end{figure}
   For an irrational flow, the velocity potential $\varphi$ s be the velocity potential satisfies:
   \begin{equation*}
     \mathbf{v}=-\nabla \varphi,\quad \text{and}\quad p=\rho_1\frac{\partial\varphi}{\partial t}
      \end{equation*}
      Then, the wave equation for $\varphi$ takes the form:
      \begin{equation}\label{eq:acoustic}
      \frac{1}{c^2}\frac{\partial^2 \varphi}{\partial t^2}-\Delta \varphi = 0, \qquad(\mathbf{x},t)\in\Omega^{+}\times J.
      \end{equation}
      \par{}The elastic wave in $\Omega$ is described by the displacement field $\mathbf{u}={(u_{1},u_{2})}^{T}$, which satisfies the dynamic linear elastic equation:
      \begin{equation}
        \label{eq:elastic}
        \begin{split}
            \rho_2 \frac{\partial^2 \mathbf{u}}{\partial t^2}-\text{div}(\sigma(\mathbf{u}))&= \mathbf{0},\qquad(\mathbf{x},t)\in\Omega\times J\\
      \sigma(\mathbf{u})&= \lambda \text{tr}(\varepsilon(\mathbf{u}))\mathbf{I}+2\mu\varepsilon(\mathbf{u})\\
      \varepsilon(\mathbf{u})&=\frac{1}{2}(\nabla \mathbf{u}+(\nabla \mathbf{u})^T).
    \end{split}
      \end{equation}
      Here $\rho_2$ is the constant density of the elastic body, which is assumed to be homogeneous and isotropic with the Lam\'{e} constants $\lambda$ and $\mu$ such that $3\lambda+2\mu>0$ and $\mu>0$.
      \par{}The velocity potential $\varphi$ in equation~\eqref{eq:acoustic} and the elastic displacement field $\mathbf{u}$ in equation~\eqref{eq:elastic} are coupled by the transmission conditions on $\Gamma\times J$, together with the homogeneous initial conditions, we can get an initial-boundary value problem:
      \begin{align}
      \rho_2 \frac{\partial^2 \mathbf{u}}{\partial t^2}-\text{div}(\sigma(\mathbf{u}))&= \mathbf{0}, &&\text{in~}\Omega\times J,\label{eq:ysfsi}\\
      \frac{1}{c^2}\frac{\partial^2 \varphi}{\partial t^2}-\Delta \varphi &= 0, &&\text{in~}\Omega^{+}\times J,\label{eq:ysfsi2}\\
      \sigma(\mathbf{u})\mathbf{n}&=-\rho_1(\frac{\partial \varphi}{\partial t}+\frac{\partial \varphi^i}{\partial t})\mathbf{n}, && \text{on~} \Gamma\times J,  \\
      \frac{\partial \varphi}{\partial\mathbf{n}} &= -\frac{\partial \varphi^i}{\partial\mathbf{n}}-\frac{\partial\mathbf{u}}{\partial t}\cdot\mathbf{n},&& \text{on~} \Gamma\times J,\\
      r^{\frac{1}{2}}\left(\frac{\partial \varphi}{\partial\mathbf{n}}+\frac{\partial \varphi}{\partial t}\right)&\rightarrow 0\quad\text{as~} r = |\mathbf{x}| \rightarrow \infty,&& \text{a.e.~}t\in J,\label{eq:ysfsi5} \\
      \mathbf{u}|_{t=0}=\left.\frac{\partial\mathbf{u}}{\partial t}\right|_{t=0}=\mathbf{0}, & \quad x \in \Omega \qquad  \text{~and~}\qquad \varphi|_{t=0}=\left.\frac{\partial \varphi}{\partial t}\right|_{t=0}=0,&&\, x \in \Omega^{+}.
      \end{align}
      Here $\mathbf{n}$ is the unit outward vector on $\Gamma$ from $\Omega$ towards $\Omega^+$, $\varphi^i$ is the given incident field and equation~\eqref{eq:ysfsi5} is the radiation condition for $\varphi$.

      \paragraph{}Before the discussion, we first introduce the definitions of some relevant Sobolev spaces and norms.  Let $L^2(\Omega)$ be the function space consisting of all square integrable functions over $\Omega$ equipped with the norm
      \[
        \|u\|_{0,\Omega}=\left(\int_{\Omega}|u(\mathbf{x},t)|^2 \,d\mathbf{x}\right)^{1/2}.
      \]
      For $s > 0$, the standard Sobolev space is denoted by
    \[
    H^{s}(\Omega) = \left\{D^\alpha u\in L^2(\Omega) \text{~for all~} |\alpha| \leq s\right\}
  \]
   with the norm
    \[
      \|u\|_{s,\Omega}=\left(\int_{\Omega}\sum\limits_{|\alpha|<s}|D^\alpha u|^2\,dx\right)^{1/2},
    \]
    and $H^s(\Gamma)$ the trace functional space for $\Gamma = \partial\Omega$ under the $L^2(\Gamma)$ inner produce
    \[
      \langle u, v\rangle_{\Gamma} = \int_\Gamma u v\,\gamma.
    \]
    It is clear to note that $H^{-s}(\Omega)$ and $H^{-s}(\Gamma)$ are the dual space of $H^{s}(\Omega)$ and $H^{s}(\Gamma)$.
    We also denote $\mathbf{L}^2(\Omega)={(L^2(\Omega))}^2$ and $\mathbf{H}^1(\Omega)={(H^1(\Omega))}^2$.  For any $\mathbf{u} = {(u_1,u_2)}^T\in\mathbf{H}^1(\Omega)$, the Frobenius norm is defined as:
      \[
      \|\nabla\mathbf{u}\|_{F(\Omega)}={\left(\sum\limits_{j=1}^2\int_{\Omega}|\nabla u_j|^2\,d\mathbf{x}\right)}^{1/2}.
      \]
      And a simple calculation gives:
      \begin{equation}
        \label{eq:Frobenius}
      \|\nabla\mathbf{u}\|^2_{F(\Omega)}+\|\text{div~}\mathbf{u}\|^2_{0,\Omega}\leq C\|\mathbf{u}\|^2_{1,\Omega},
      \end{equation}
      where $C$ is a positive constant.
      \par{}Furthermore, for $1\leq q\leq\infty$ we will make use of the Bochner space $L^q(J;H^1(\Omega))$, endowed with the norm:
      \[
        \|w(x,t)\|_{L^q(J;H^1(\Omega))}=\left\{\begin{array}{cc}
            {\left(\int_J\|w\|^q_{1,\Omega}\right)}^{1/q}& 1\leq q<\infty, \\
      \text{ess}\sup\limits_{0\leq t\leq T}\|w\|_{1,\Omega}\,dt &
  q=\infty,\end{array}\right.
      \]
      and the Bochner space $H^1(J; H^s(\Omega))$, endowed with norm:
        \[
          \|w(x,t)\|_{H^1(J;H^s(\Omega))} = {\left(\int_J {\left(\|w\|_{s,\Omega}^2 + \|w_t\|_{s,\Omega}^2\,dt\right)}\right)}^{1/2}.
        \]

\section{The reduced problem in bounded domain}
   \paragraph{}To reduce this exterior problem to a problem in a bounded domain, we impose the first order approximate boundary condition on the artificial boundary $\Gamma_R$ (see  Figure~\ref{fig:domain}):
   \begin{figure}
     \centering
     \includegraphics[height=7.0cm,width=7.5cm]{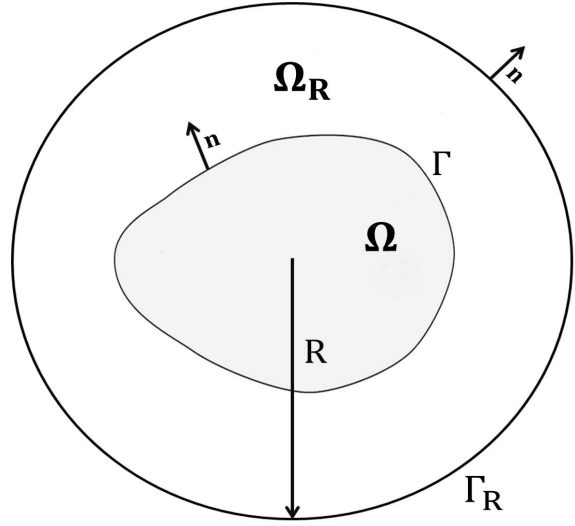}
     \caption{Bounded domain}\label{fig:domain}
   \end{figure}
   \[
     \frac{\partial \varphi}{\partial \mathbf{n}}=-\frac{\partial \varphi}{\partial t},\qquad \text{~on~}\Gamma_R \times J,
    \]
      where $\mathbf{n}$ is the outward unit vector from $\Omega_R$ on $\Gamma_R$. Then the reduced problem on the bounded domain reads:
      \begin{align}
      \rho_2 \frac{\partial^2 \mathbf{u}}{\partial t^2}-\text{div}(\sigma(\mathbf{u}))&= \mathbf{0}, &&\text{~in~} \Omega\times J,\label{eq:rdfsi1}\\
      \frac{1}{c^2}\frac{\partial^2 \varphi}{\partial t^2}-\Delta \varphi &= 0 ,&& \text{~in~} \Omega_{R}\times J,\label{eq:rdfsi2}\\
       \sigma(\mathbf{u})\mathbf{n}&= -\rho_1(\frac{\partial \varphi}{\partial t}+\frac{\partial \varphi^i}{\partial t})\mathbf{n}, && \text{~on~} \Gamma\times J,\label{eq:rdfsi3}\\
       \frac{\partial \varphi}{\partial\mathbf{n}} &= -\frac{\partial \varphi^i}{\partial\mathbf{n}}-\frac{\partial\mathbf{u}}{\partial t}\cdot\mathbf{n}, &&\text{~on~} \Gamma\times J,\label{eq:rdfsi4}\\
      \frac{\partial \varphi}{\partial \mathbf{n}}&= -\frac{\partial \varphi}{\partial t}, && \text{~on~} \Gamma_R\times J,\label{eq:rdfsi5}\\
      \mathbf{u}|_{t=0}=\left.\frac{\partial\mathbf{u}}{\partial t}\right|_{t=0} & =\mathbf{0},\,x \in \Omega, \qquad\quad\text{and} && \varphi|_{t=0}=\left.\frac{\partial \varphi}{\partial t}\right|_{t=0}=0,\, x \in \Omega_R. \label{eq:rdfsi6}
      \end{align}
   Then, we present the well-posedness and stability of the reduced problem in the follows.
   \subsection{Well-posedness and stability}
   \par{}First of all, we need to show that the reduced interaction problem in the bounded domain is well-posed and stable.
      \begin{theorem}\label{thm31}
       {Suppose the incident wave $\varphi^i\in H^1(J;L^2(\Omega_R))$, $\varphi^i|_{t=0}=0$ and $spt(\varphi^i)\subset\Omega_R\times J$. Then~\eqref{eq:rdfsi1}-\eqref{eq:rdfsi6} has a unique solution satisfying:
         \[
           \mathbf{u}\in L^2\left(J;\mathbf{H^1}(\Omega)\right)\cap H^1\left(J;\mathbf{L^2}(\Omega)\right),
         \]
         \[
           \varphi\in L^2\left(J;H^1(\Omega_R)\right)\cap H^1\left(J;L^2(\Omega_R)\right),
         \]
         and we have the following stability estimate:
      \begin{align}
      \nonumber
      &\,\max\limits_{t \in [0,T]}\left(\|\mathbf{\mathbf{u}}_{t}\|_{0,\Omega}+\|\nabla\mathbf{u}\|_{F(\Omega)}+\|\nabla\cdot\mathbf{u}\|_{0,\Omega}\right)\\
      \nonumber
      \leq&\, C\left(\left\|\varphi^i_t\right\|^2_{L^1(J;H^{-1/2}(\Gamma))}+\left\|\nabla\varphi^i\right\|^2_{L^1(J;H^{-1/2}(\Gamma))}+\max_{t\in[0,T]}\left\|\varphi_{tt}^i\right\|^2_{-\frac{1}{2},\Gamma} \right.\\
       &\qquad \left. +\max_{t\in[0,T]}\left\|\nabla\varphi_{t}^i\right\|^2_{-\frac{1}{2},\Gamma}
       +\|\varphi_{ttt}^i\|^2_{L^1(J;H^{-1/2}(\Gamma))}+\|\nabla\varphi_{tt}^i\|^2_{L^1(J;H^{-1/2}(\Gamma))}\right), \label{eq:stability1.1}
    \end{align}
    and
    \begin{align}
      \nonumber
      &\, \max\limits_{t\in[0,T]}\left(\|\varphi_{t}\|_{0,\Omega_R}+\|\nabla \varphi\|_{0,\Omega_R}\right) \\
      \leq&\, C\left(\left\|\varphi^i_t\right\|^2_{L^1(J;H^{-1/2}(\Gamma))}+\left\|\nabla\varphi^i\right\|^2_{L^1(J;H^{-1/2}(\Gamma))}+\max_{t\in[0,T]}\left\|\varphi_{tt}^i\right\|^2_{-\frac{1}{2},\Gamma} \right.\\
       &\qquad \left. +\max_{t\in[0,T]}\left\|\nabla\varphi_{t}^i\right\|^2_{-\frac{1}{2},\Gamma}
       +\|\varphi_{ttt}^i\|^2_{L^1(J;H^{-1/2}(\Gamma))}+\|\nabla\varphi_{tt}^i\|^2_{L^1(J;H^{-1/2}(\Gamma))}\right). \label{eq:stability1.2}
      \end{align}
      }
      \end{theorem}
Although different boundary conditions are considered to develop the reduced problem, the proof of above theorem~\ref{thm31} is analogous to the proof of Theorem 3.2 in~\cite{BGL2018}.   We omit to present the proof of Theorem~\ref{thm31} in this section, and report the detailed proof of Theorem~\ref{thm31} in the appendix for the sake of completeness.

   \subsection{Variational formulation}
   \paragraph{}The standard variational formulation of problem \eqref{eq:rdfsi1}-\eqref{eq:rdfsi6} is as follows: For given incident wave $\varphi^i$, find
\[
  (\mathbf{u},\varphi)\in L^2\left(J;\mathbf{H^1}(\Omega)\right)\times L^2\left(J;H^1(\Omega_R)\right),
\]
with
\[
  \mathbf{u}_{t}\in L^2\left(J;\mathbf{L^{2}}(\Omega)\right),\,\varphi_{t}\in L^2\left(J;L^{2}(\Omega_R)\right),\, \mathbf{u}_{tt}\in L^2\left(J;\mathbf{H}^{-1}(\Omega)\right),\,\varphi_{tt}\in L^2\left(J;H^{-1}(\Omega_R)\right),
\]
such that, $\forall~\mathbf{v}\in \mathbf{H^1}(\Omega)\forall~\phi\in H^1(\Omega_R)$:
   \begin{eqnarray}\label{eq:VF1}
     \frac{\rho_2}{\rho_1}\left\langle\mathbf{u}_{tt},\mathbf{v}\right\rangle+\frac{1}{c^2}\left\langle\varphi_{tt},\phi\right\rangle+a(\mathbf{u},\varphi;\mathbf{v},\phi)=L(\mathbf{v},\phi).
   \end{eqnarray}
with the initial conditions
\[
  \mathbf{u}|_{t=0}=\mathbf{u}_t|_{t=0}=\mathbf{0},\quad \varphi|_{t=0}=\varphi_t|_{t=0}=0.
\]
Here $\langle\cdot,\cdot\rangle$ denotes the duality pairing between spaces $H^{-1}$ and $H^0$ with the associated domain, and $a(\mathbf{u},\varphi;\mathbf{v},\phi)$ and $L(\mathbf{v},\phi)$ are defined as:
   \begin{eqnarray}\label{eq:VF2}
    \nonumber
     a(\mathbf{u},\varphi;\mathbf{v},\phi)&=&\frac{\lambda}{\rho_1}(\nabla\cdot\mathbf{u},\nabla\cdot \mathbf{v})_{0,\Omega}+\frac{2\mu}{\rho_1}\left(\varepsilon(\mathbf{u}):\varepsilon(\mathbf{v})\right)_{F(\Omega)}+\left(\varphi_t\mathbf{n},\mathbf{v}\right)_{0,\Gamma}\\
     &&+(\nabla\varphi,\nabla\phi)_{0,\Omega_R}-(\mathbf{u}_t\cdot\mathbf{n},\phi)_{0,\Gamma} +(\varphi_t,\phi)_{0,\Gamma_R},\\
     L(\mathbf{v},\phi)&=&\left(\frac{\partial\varphi^i}{\partial\mathbf{n}},\phi\right)_{0,\Gamma}-\left(\varphi_t^i\mathbf{n},\mathbf{v}\right)_{0,\Gamma}
   \end{eqnarray}
   with $(\cdot,\cdot)$ the standard inner product in $L^2$-space.

\section{The IPDG method}
   \subsection{Spaces, jumps and averages}
   \paragraph{}Assume that $\Omega$ with Lipschitz boundary $\Gamma$ is regularly divided into disjoint elements ${E}$ by mesh ${\mathbf{\mathcal{E}_h}}$ such that $\overline{\Omega}=\bigcup_{E\in\mathbf{\mathcal{E}_h}} \overline{E}$, where $E$ is a triangle or quadrilateral in 2D, or a tetrahedron or hexahedron in 3D.  Similarly, The regular meshes $\mathcal{K}_h$ partitions the fluid domain $\Omega_R$ into disjoint elements $K$ such that $\overline{\Omega}_R =\bigcup_{K\in\mathcal{K}_h}\overline{K}$. The diameter of element $E/K$ is denoted by $h_{E/K}$, and $h$ is the mesh size given by $h=\max\limits_{E\in \mathbf{\mathcal{E}}_h,K\in\mathcal{K}_h}(h_E,h_K)$.
      \par{}The boundary of the elastic body $\Omega$, i.e., $\Gamma$ is approximated by the boundary edges of the subdivision: The set of boundary edges of $\mathbf{\mathcal{E}}_h$ is $\Gamma_h:=\partial E\bigcap \Gamma$, since $\Omega\bigcap\Omega_R=\Gamma$, the edges in $\Gamma_h$ are also the boundary edges of $\mathcal{K}_h$, so the sets of boundary edges of $\mathcal{K}_h$ are $\Gamma_h\bigcup\Gamma_h^R$, $\Gamma_h^R:=\partial K\bigcap\Gamma_R$. We denote by $\Gamma_I^1$ and $\Gamma_I^2$ the set of interior edges of the subdivision $\mathbf{\mathcal{E}}_h$ and $\mathcal{K}_h$ respectively.
      \par{}For any real number $m$, the broken Sobolev spaces $H^m(\mathcal{E}_h)$ and $H^m(\mathcal{K}_h)$ are defined as :
      \[
        \mathbf{H}^m(\mathbf{\mathcal{E}}_h)=\left\{\mathbf{v}\in \mathbf{L^2}(\Omega):\forall E\in \mathbf{\mathcal{E}}_h,\mathbf{v}|_E\in \mathbf{H}^m(E)\right\},
      \]
      and
      \[
        H^m(\mathcal{K}_h)=\left\{\varphi\in L^2(\Omega_R):\forall K\in\mathcal{K}_h,\varphi|_K\in H^m(K)\right\},
      \]
      with the broken Sobolev norms:
      \[
        \|\mathbf{v}\|_{m,\mathcal{E}_h}=\left(\sum\limits_{E\in\mathcal{E}_h}\|\mathbf{v}\|^2 _{m,E}\right)^{\frac{1}{2}},\quad|\varphi\|_{m,\mathcal{K}_h}=\left(\sum\limits_{K\in\mathcal{K}_h} \|\varphi\|^2_{m,K}\right)^{\frac{1}{2}}.
      \]
      In particular, we also define the broken gradient seminorm:
      \[
        \|\nabla \mathbf{v}\|_{0,\mathcal{E}_h}=\left(\sum\limits_{E\in\mathcal{E}_h}\|\nabla \mathbf{v}\|^2_{0,E}\right)^{\frac{1}{2}},~\|\nabla \varphi\|_{0,\mathcal{K}_h}=\left(\sum\limits_{K\in\mathcal{K}_h}\|\nabla \varphi\|^2_{0,K}\right)^{\frac{1}{2}}.
      \]
      \par{}If $\varphi\in H^m(\mathcal{K}_h)$, the trace of $\varphi$ along any side of each element $K \in \mathcal{K}_h$ is well defined.  Let $e$ be the edge between the elements $K_1$ and $K_2$, then the jump and average of a scalar function $\varphi$ on $e$ are given by:
      \[
        \{\varphi\}=\frac{1}{2}\left(\varphi|_{K_1}+\varphi|_{K_2}\right),\quad \text{and}\quad [\varphi]=\varphi|_{K_1}\bn_{K_1}+ \varphi|_{K_2}\bn_{K_2},
      \]
      respectively, where $\bn_{K_1}$ is the outward unit normal from $K_1$ to $K_2,$ and likewise for $\bn_{K_2}$.  On the boundary edge $e\in\Gamma_h$ or $e\in\Gamma_h^R$, we extend the definition: $\{\varphi\}=[\varphi]=\varphi|_{K^e}$, where $K_e$ is the element that $K_e\bigcap\Gamma_h=e$ or $K_e\bigcap\Gamma_h^R=e$.

      \par{}Similarly, for a vector valued function $\mathbf{v}\in \mathbf{H}^s(\mathbf{\mathcal{E}}_h)$, the jump and average of a vector function $\mathbf{v}$ on $e$ are given by:
      \[
        \{\mathbf{v}\}=\frac{1}{2}(\mathbf{v}|_{E_1}+\mathbf{v}|_{E_2}),\quad \text{and} \quad [\mathbf{v}]=\mathbf{v}|_{E_1}\cdot\mathbf{n}_{E_1}+ \mathbf{v}|_{E_2}\cdot\mathbf{n}_{E_2},
      \]
      respectively, where $\mathbf{n}_{E_1}$ is outward unit normal from $E_1$ to $E_2,$ and likewise for $\mathbf{n}_{E_2}$. Similarly, on the boundary edge $e$: $\{\mathbf{v}\}=[\mathbf{v}]=\mathbf{v}|_{E^e}$, where $E_e$ is the element that $E_e\bigcap\Gamma_h=e$.

   \subsection{Spatial discretization}
   \paragraph{}For given partitions $\mathbf{\mathcal{E}}_h$ of $\Omega$, $\mathcal{K}_h$ of $\Omega_R$ and an approximation order $k\geq 1$, we can approximate the solution $(\mathbf{u},\varphi)$ of~\eqref{eq:rdfsi1}-\eqref{eq:rdfsi6} in the finite element subspaces
      \[
        \mathbf{\mathcal{D}}_k(\mathbf{\mathcal{E}}_h):=\left\{\mathbf{v}\in \mathbf{L^2}(\Omega):~\forall E\in \mathbf{\mathcal{E}}_h,~v_i|_E\in\mathbb{P}_k(E),i=1,2\right\},
      \]
and
      \[
        \mathcal{D}_k(\mathcal{K}_h):=\left\{\varphi\in \mathbf{L^2}(\Omega_R):~\forall K\in \mathcal{K}_h,~\varphi|_K\in\mathbb{P}_k(K)\right\},
      \]
      of the broken Sobolev space $H^m(\mathcal{E}_h)$ and $H^m(\mathcal{F}_h)$ for $m>3/2$.  Here $\mathbb{P}_k$ denotes the space of polynomials of degree at most $k$.

      \par{}Then we consider the following semidiscrete DG approximation of~\eqref{eq:rdfsi1}-\eqref{eq:rdfsi6}: find $\mathbf{u}^h \in \mathbf{\mathcal{D}}_k(\mathbf{\mathcal{E}}_h)$ and $\varphi^h \in \mathcal{D}_k(\mathcal{K}_h)$ such that
      \begin{eqnarray}\label{eq:IPDGVF1}
     \frac{\rho_2}{\rho_1}\left\langle\mathbf{u^h}_{tt},\mathbf{v}\right\rangle+\frac{1}{c^2}\left\langle\varphi^h_{tt},\phi\right\rangle+a_h(\mathbf{u}^h,\varphi^h;\mathbf{v},\phi)=L_h(\mathbf{v},\phi),
   \end{eqnarray}
   holds for any $\mathbf{v} \in \mathbf{\mathcal{D}}_k(\mathbf{\mathcal{E}}_h)$ and $\phi \in \mathcal{D}_k(\mathcal{K}_h)$.  Here, the discrete bilinear form $a_h(\mathbf{u}^h,\varphi^h;\mathbf{v},\phi)$ and linear form $L_h(\mathbf{v},\phi)$ are given by the IPDG discretization as:
   \begin{eqnarray}\label{eq:IPDG2}
    \nonumber
    a_h(\mathbf{u}^h,\varphi^h;\mathbf{v},\phi)&=&\sum\limits_{E\in\mathbf{\mathcal{E}}_h}\left(\frac{\lambda}{\rho_1}(\nabla\cdot\mathbf{u}^h,\nabla\cdot \mathbf{v})_{0,E}+\frac{2\mu}{\rho_1}\left(\varepsilon(\mathbf{u}^h):\varepsilon(\mathbf{v})\right)_{F(E)}\right)\\
     \nonumber
     &&\,\,-\sum\limits_{e\in \Gamma_I^1}\int_e\{\sigma(\mathbf{u}^h)\cdot\mathbf{n}\}[\mathbf{v}]\,dS-\sum\limits_{e\in \Gamma_I^1}\int_e\{\sigma(\mathbf{v})\cdot\nu\}[\mathbf{u}^h]\,dS
     \\
     \nonumber
     &&\,\,+\sum\limits_{K\in\mathcal{K}_h}(\nabla \varphi^h,\nabla\phi)_{0,K}-\sum\limits_{e\in \Gamma_I^2}\int\limits_e\left\{\frac{\partial \varphi^h}{\partial\mathbf{n}}\right\}[\phi]\,dS-\sum\limits_{e\in \Gamma_I^2}\int_e\left\{\frac{\partial \phi}{\partial\mathbf{n}}\right\}[\varphi^h]\,dS\\
     \nonumber
     &&\,\,+\sum\limits_{e\in\Gamma_I^2} \int_e \frac{\alpha}{|e|^{\beta}}[\varphi^h][\phi]\,dS+\sum\limits_{e\in\Gamma_I^1} \int_e \frac{\alpha}{|e|^{\beta}}[\mathbf{u}^h][\mathbf{v}]\,dS+\sum\limits_{e\in\Gamma_h}\int_e\varphi^h_t\mathbf{n}\mathbf{v}\,dS\\
     &&\,\,-\sum\limits_{e\in\Gamma_h}\int_e\mathbf{u}^h_t\cdot\mathbf{n}\phi\, dS+\sum\limits_{e\in\Gamma_h^R} \int_{e}\varphi^h_t\phi \,dS\\
     L_h(\mathbf{v},\phi)&=& \sum\limits_{e\in\Gamma_h}\int_e\frac{\partial\varphi^i}{\partial\mathbf{n}}\phi\, dS-\sum\limits_{e\in\Gamma_h}\int_e\varphi_t^i\mathbf{n}\mathbf{v}\,dS.
   \end{eqnarray}

   \par{}The interior penalty stabilization function $\frac{\alpha}{|e|^{\beta}}$ penalizes the jumps over the edges of $\mathbf{\mathcal{E}}_h$ and $\mathcal{K}_h$, where $\alpha$ is a positive parameter independent of the local mesh sizes. Here $|e|$ simply means the length of $e$ and we have
   \[
     |e|\leq h_{E/K}\leq h,\quad \forall e\in\partial E \text{~or~} \partial K,
   \]
   with $\beta > 0$, and when $\beta > 1$ the method is superpenalized.

   \par{}We define the space $\mathbf{V}_1(h)=\mathbf{H}^1(\Omega)+\mathbf{\mathcal{D}}_k(\mathbf{\mathcal{E}}_h)$ and $V_2(h)=H^1(\Omega_R)+\mathcal{D}_k(\mathcal{K}_h)$. On $\mathbf{V}_1(h)$ and $V_2(h)$, we define the DG energy norm as
   \[
     \|\mathbf{v}\|_h^2 := \sum\limits_{E\in\mathbf{\mathcal{E}}_h}\left(\left\|\sqrt{\frac{2\mu}{\rho_1}}\varepsilon(\mathbf{v})\right\|^2_{F(E)}+\left\|\sqrt{\frac{\lambda}{\rho_1}}\nabla\cdot\mathbf{v}\right\|^2_{0,E}\right) +\frac{1}{h}\sum\limits_{e\in\Gamma_h\cup\Gamma_I^1}\left\|[\mathbf{v}]\right\|^2_{0,e},
   \]
   and
   \[
     \|\phi\|_h^2 := \sum\limits_{K\in\mathcal{K}_h}\|\nabla\varphi\|_{0,K}^2 + \frac{1}{h}\sum\limits_{e\in\Gamma_h\cup\Gamma_I^2}\|[\phi]\|^2_{0,e}.
\]

\paragraph{}The consistency of the scheme is straightforward since the jumps at element boundaries vanishes when $\mathbf{u},\mathbf{v}\in \mathbf{H}^m(\mathbf{\mathcal{E}}_h)$ and $\varphi,\phi\in H^m(\mathcal{K}_h)$ for $m>2/3$. Next we will discuss the property of the bilinear form $ a_h(\cdot,\cdot;\cdot,\cdot)$.  For more details about the IPDG method, one can refer to~\cite{MJGASDS2006}

\begin{lemma}[Coercivity]\label{lemma4.1}{There exists a positive constant $C_{\text{coer}}$ independent of $h$ such that for all $\mathbf{v}\in\mathbf{V}_1(h)$ and $\phi \in V_2(h)$
    \begin{eqnarray}\label{eq:coericity}
    \nonumber
    a_h(\mathbf{v},\phi;\mathbf{v},\phi)&\geq& C_{\text{coer}}\left(\sum\limits_{E\in\mathbf{\mathcal{E}}_h}\|\varepsilon(\mathbf{v})\|_{F(E)}^2+\sum\limits_{e\in\Gamma_h\cup\Gamma^1_I}\|[\mathbf{v}]\|_{0,e}^{2}\right.\\ &&\quad\left.+\sum\limits_{K\in\mathcal{K}_h}\|\nabla\phi\|_{0,K}^{2}+\sum\limits_{e\in\Gamma_h\cup\Gamma^2_I}\|[\phi]\|_{0,e}^2\right).
    \end{eqnarray}
    }
    \end{lemma}
    {\bf Proof:} By Cauchy-Schwarz inequality,
    \begin{eqnarray*}
      \sum\limits_{e\in \Gamma_I^1}\int\limits_e\{\sigma(\mathbf{v})\cdot\mathbf{n}\}[\mathbf{v}]\,dS &\leq&\sum\limits_{e\in \Gamma_I^1}\|\{\sigma(\mathbf{v})\cdot\mathbf{n}\}\|_{0,e}\|[\mathbf{v}]\|_{0,e} \\
    &\leq&\sum\limits_{e\in \Gamma_I^1}\|\{\sigma(\mathbf{v})\cdot\mathbf{n}\}\|_{0,e}\left(\frac{1}{|e|^{\beta}}\right)^{\frac{1}{2}-\frac{1}{2}}\|[\mathbf{v}]\|_{0,e}.
    \end{eqnarray*}
    Consider the average of the fluxes for an interior edge $e$ shared by $E_1^e$ and $E_2^e$ and apply the trace theorem, we have:
    \begin{eqnarray*}
    \|\{\sigma(\mathbf{v})\cdot\mathbf{n}\}\|_{0,e}&\leq&\frac{1}{2}\|(\sigma(\mathbf{v})\cdot\mathbf{n})|_{E_1^e}\|_{0,e}+ \frac{1}{2}\|(\sigma(\mathbf{v})\cdot\mathbf{n})|_{E_2^e}\|_{0,e} \\
    &\leq&\frac{C_+}{2}h_{E_1^e}^{-\frac{1}{2}}\|\sigma(\mathbf{v})\|_{F(E_1^e)} +\frac{C_+}{2}h_{E_1^e}^{-\frac{1}{2}}\|\sigma(\mathbf{v})\|_{F(E_2^e)}.
    \end{eqnarray*}
    Assume $\beta\geq1$ and $h\leq1$, we have:
    \begin{eqnarray*}
      \int_e\{\sigma(\mathbf{v})\cdot\mathbf{n}\}[\mathbf{v}]\,dS&\leq&\frac{C_+}{2}|e|^{\frac{\beta}{2}} \bigg(h_{E_1^e}^{-\frac{1}{2}}\|\sigma(\mathbf{v})\|_{F(E_1^e)}+h_{E_1^e}^{-\frac{1}{2}} \|\sigma(\mathbf{v})\|_{F(E_2^e)}\bigg)\left(\frac{1}{|e|^{\beta}}\right)^{\frac{1}{2}} \|[\mathbf{v}]\|_{0,e}\\
    &\leq&\frac{C_+}{2}\bigg(h_{E_1^e}^{\frac{\beta}{2}-\frac{1}{2}} +h_{E_1^e}^{\frac{\beta}{2}-\frac{1}{2}}\bigg)\bigg(\|\sigma(\mathbf{v})\|^2_{F(E_1^e)}+\|\sigma(\mathbf{v})\|^2_{F(E_2^e)}\bigg)^{\frac{1}{2}} \left(\frac{1}{|e|^{\beta}}\right)^{\frac{1}{2}} \|[\mathbf{v}]\|_{0,e}\\
    &\leq&C_+\bigg(\|\sigma(\mathbf{v})\|^2_{F(E_1^e)}+\|\sigma(\mathbf{v})\|^2_{F(E_2^e)}\bigg)^{\frac{1}{2}} \left(\frac{1}{|e|^{\beta_0}}\right)^{\frac{1}{2}} \|[\mathbf{v}]\|_{0,e}.
    \end{eqnarray*}
    We denote $n_0$ the maximum number of neighbors an element can have.
     \begin{eqnarray*}
       \sum\limits_{e\in \Gamma_I^1}\int_e\{\sigma(\mathbf{v})\cdot\mathbf{n}\}[\mathbf{v}]\,dS
    &\leq&C_+\bigg(\sum\limits_{e\in \Gamma_I^1}\big(\|\sigma(\mathbf{v})\|^2_{F(E_1^e)}+\|\sigma(\mathbf{v})\|^2_{F(E_2^e)}\big)\bigg)^{\frac{1}{2}}\bigg(\sum\limits_{e\in \Gamma_I^1}\frac{1}{|e|^{\beta}} \|[\mathbf{v}]\|^2_{0,e}\bigg)^{\frac{1}{2}}\\
    &\leq&C_+\sqrt{n_0}\bigg(\sum\limits_{E\in \mathcal{E}_h}\|\sigma(\mathbf{v})\|^2_{F(E)}\bigg)^{\frac{1}{2}}\bigg(\sum\limits_{e\in \Gamma_I^1}\frac{1}{|e|^{\beta}} \|[\mathbf{v}]\|^2_{0,e}\bigg)^{\frac{1}{2}}.
    \end{eqnarray*}
    Using Young's inequality, we have for $\delta>0$
    \begin{eqnarray*}
      \sum\limits_{e\in \Gamma_I^1}\int_e\{\sigma(\mathbf{v})\cdot\mathbf{n}\}[\mathbf{v}]\,dS
    &\leq&\frac{\delta}{2}\sum\limits_{E\in \mathcal{E}_h}\|\sigma(\mathbf{v})\|^2_{F(E)}+\frac{C_+^2n_0}{2\delta}\sum\limits_{e\in \Gamma_I^1}\frac{1}{|e|^{\beta}} \|[\mathbf{v}]\|^2_{0,e}.
    \end{eqnarray*}
    Following a similar steps, we get:
    \begin{eqnarray*}
    \sum\limits_{e\in \Gamma_I^2}\int_e\left\{\frac{\partial\phi}{\partial\mathbf{n}}\right\}[\phi]dS&\leq&\frac{\delta}{2}\sum\limits_{K\in \mathcal{K}_h}\|\nabla \phi\|^2_{0,K}+\frac{C_+^2n_0}{2\delta}\sum\limits_{e\in \Gamma_I^2}\frac{1}{|e|^{\beta}} \|[\phi]\|^2_{0,e}.
    \end{eqnarray*}
    Thus we obtain a lower bound for $a_h(\mathbf{v},\phi;\mathbf{v},\phi)$:
    \begin{eqnarray*}
      a_h(\mathbf{v},\phi;\mathbf{v},\phi)&\geq&\sum\limits_{E\in\mathbf{\mathcal{E}}_h}\left(\frac{\lambda}{\rho_1}\|\nabla\cdot \mathbf{v}\|^2_{0,E}+\frac{2\mu}{\rho_1}\|\varepsilon(\mathbf{v})\|^2_{F(E)}-\delta\|\sigma(\mathbf{v})\|^2_{F(E)}\right)\\
    &&+\sum\limits_{e\in \Gamma_I^1}\frac{\alpha\delta-C_+^2n_0}{\delta|e|^{\beta}} \|[\mathbf{v}]\|^2_{0,e}+\sum\limits_{K\in\mathcal{K}_h}(1-\delta)\|\nabla\phi\|^2_{0,K}+\sum\limits_{e\in \Gamma_I^2}\frac{\alpha\delta-C_+^2n_0}{\delta|e|^{\beta}} \|[\phi]\|^2_{0,e}\\
    &&-C\sum\limits_{e\in \Gamma_h}\left(\|\phi_t\|^2_{0,e}+\|\mathbf{v}\|^2_{0,e}-\|\phi\|^2_{0,e}-\|\mathbf{v}_t\|^2_{0,e}\right)-C\sum\limits_{e\in \Gamma_h^R}\left(\|\phi_t\|^2_{0,e}+\|\phi\|^2_{0,e}\right).
    \end{eqnarray*}
  For any $0<t\leq T$, using the Young's inequality, similar as~\eqref{eq:stable12} we have
       \begin{equation*}
         \|\mathbf{v}\|_{0,e}^2\leq C\|\mathbf{v}_{t}\|^2_{0,e}\quad \|\phi\|_{0,e}^2\leq C\|\phi_{t}\|^2_{0,e}.
       \end{equation*}
       Then, we have
    \begin{eqnarray*}
      a_h(\mathbf{v},\phi;\mathbf{v},\phi)&\geq&\sum\limits_{E\in\mathbf{\mathcal{E}}_h}\left(\frac{\lambda}{\rho_1}\|\nabla\cdot \mathbf{v}\|^2_{0,E}+\frac{2\mu}{\rho_1}\|\varepsilon(\mathbf{v})\|^2_{F(E)}-\delta\|\sigma(\mathbf{v})\|^2_{F(E)}\right)\\
    &&+\sum\limits_{e\in \Gamma_I^1}\frac{\alpha\delta-C_+^2n_0}{\delta|e|^{\beta}} \|[\mathbf{v}]\|^2_{0,e}+\sum\limits_{K\in\mathcal{K}_h}(1-\delta)\|\nabla\phi\|^2_{0,K}+\sum\limits_{e\in \Gamma_I^2}\frac{\alpha\delta-C_+^2n_0}{\delta|e|^{\beta}} \|[\phi]\|^2_{0,e}\\
    &&+C\sum\limits_{e\in \Gamma_h}\left(\|\phi\|^2_{0,e}+\|\mathbf{v}\|^2_{0,e}\right).
    \end{eqnarray*}
    Since $\sigma(\mathbf{v})=(\lambda\,\text{div}(\mathbf{v}))\mathrm{I}+2\mu\varepsilon(\mathbf{v})$ and triangle inequality,
    $$\|\sigma(\mathbf{v})\|^2_{F(E)}\leq\lambda\|\nabla\cdot \mathbf{v}\|^2_{0,E}+2\mu\|\varepsilon(\mathbf{v})\|^2_{F(E)}.$$
    Therefore, we have
    \begin{eqnarray*}
    a_h(\mathbf{v},\phi;\mathbf{v},\phi)&\geq&\sum\limits_{E\in\mathcal{E}_h}\left(\left(\frac{\lambda}{\rho_1}-\delta\right)\|\nabla\cdot \mathbf{v}\|^2_{0,E}+\left(\frac{2\mu}{\rho_1}-\delta\right)\|\varepsilon(\mathbf{v})\|^2_{F(E)}\right)\\
    &&+\sum\limits_{e\in \Gamma_I^1}\frac{\alpha\delta-C_+^2n_0}{\delta|e|^{\beta}} \|[\mathbf{v}]\|^2_{0,e}+\sum\limits_{K\in\mathcal{K}_h}(1-\delta)\|\nabla\phi\|^2_{0,K}+\sum\limits_{e\in \Gamma_I^2}\frac{\alpha\delta-C_+^2n_0}{\delta|e|^{\beta}} \|[\phi]\|^2_{0,e}\\
    &&+C\sum\limits_{e\in \Gamma_h}\left(\|\phi\|^2_{0,e}+\|\mathbf{v}\|^2_{0,e}\right).
    \end{eqnarray*}
    From the definition of the strain tensor, $\|\varepsilon(\mathbf{v})\|_{F(E)}\leq\|\nabla\cdot \mathbf{v}\|_{0,E}$, we have
    \begin{eqnarray*}
    a_h(\mathbf{v},\phi;\mathbf{v},\phi)&\geq&\sum\limits_{E\in\mathcal{E}_h}\left(\frac{\lambda}{\rho_1}+\frac{2\mu}{\rho_1}-2\delta\right)\|\varepsilon(\mathbf{v})\|^2_{F(E)}+\sum\limits_{e\in \Gamma_I^1\cup\Gamma_h}\frac{\alpha\delta-C_+^2n_0+C}{\delta|e|^{\beta}} \|[\mathbf{v}]\|^2_{0,e}\\
    &&+\sum\limits_{K\in\mathcal{K}_h}(1-\delta)\|\nabla\phi\|^2_{0,K}+\sum\limits_{e\in \Gamma_I^2\cup\Gamma_h}\frac{\alpha\delta-C_+^2n_0+C}{\delta|e|^{\beta}} \|[\phi]\|^2_{0,e}.
    \end{eqnarray*}
    Assume $\alpha$ is large enough (e.g. $\alpha\geq \frac{C_+^2n_0+C}{\delta}$), take $C_{\text{coer}}=\lambda+2\mu-2\delta$, we obtain the coercivity result for $a_h(\mathbf{v},\phi;\mathbf{v},\phi)$.
    \qed

    \begin{lemma}[Continuity]\label{lemma4.2}{There exists a positive constant $C_{\text{cont}}$ independent of $h$ such that
    \begin{eqnarray}\label{eq:C0}
    \nonumber
    |a_h(\mathbf{u},\varphi;\mathbf{v},\phi)|&\leq& C_{\text{cont}}\left(\sum\limits_{E\in\mathbf{\mathcal{E}}_h}\|\varepsilon(\mathbf{u})\|_{F(E)}\|\varepsilon(\mathbf{v})\|_{F(E)}+\sum\limits_{e\in\Gamma_h\cup\Gamma^1_I}\|[\mathbf{u}]\|_{0,e}\|[\mathbf{v}]\|_{0,e}\right.\\ &&\left.+\sum\limits_{K\in\mathcal{K}_h}\|\nabla\varphi\|_{0,K}\|\nabla\phi\|_{0,K}+\sum\limits_{e\in\Gamma_h\cup\Gamma^2_I}\|[\varphi]\|_{0,e}\|[\phi]\|_{0,e}\right).
    \end{eqnarray}
    }
    \end{lemma}
   {\bf Proof:} Using Cauchy-Schwartz inequality, we obtain:
   \begin{eqnarray}\label{eq:C1}
     \left|\sum\limits_{E\in\mathbf{\mathcal{E}}_h} \left(\frac{\lambda}{\rho_1}(\nabla\cdot\mathbf{u}, \nabla\cdot \mathbf{v})_{0,E}+\frac{2\mu}{\rho_1}\left(\varepsilon(\mathbf{u}):\varepsilon(\mathbf{v})\right)_{F(E)}\right)\right| \leq \sum\limits_{E\in\mathbf{\mathcal{E}}_h}\frac{\lambda+8\mu}{\rho_1}\|\nabla\cdot\mathbf{u}\|_{0,E}\|\nabla\cdot\mathbf{v}\|_{0,E}.
   \end{eqnarray}
   Similarly, we have
   \begin{eqnarray}\label{eq:C2}
     \left|\sum\limits_{K\in\mathcal{K}_h}(\nabla \varphi,\nabla\phi)_{0,K}\right|\leq \sum\limits_{K\in\mathcal{K}_h}C\|\nabla\varphi\|_{0,K}\|\nabla\phi\|_{0,K}.
   \end{eqnarray}
   According to Cauchy-Schwartz inequality and the trace theorem:
   \begin{eqnarray}\label{eq:C5}
     \left|\sum\limits_{e\in\Gamma_h} \int_e\varphi_t\mathbf{n}\mathbf{v}\,dS\right| \leq\sum\limits_{e\in\Gamma_h}C\|\varphi_t\|_{-\frac{1}{2},e}\|\mathbf{v}\|_{\frac{1}{2},e} \leq\sum\limits_{e\in\Gamma_h}C\|\varphi\|_{\frac{1}{2},e}\|\mathbf{v}\|_{\frac{1}{2},e}\leq\sum\limits_{e\in\Gamma_h}C\|\varphi\|_{0,e}\|\mathbf{v}\|_{0,e},
   \end{eqnarray}
   \begin{eqnarray}\label{eq:C3}
     \left|\sum\limits_{e\in\Gamma_h}\int_e\mathbf{u}_t\cdot\mathbf{n}\phi\,dS\right|\leq \sum\limits_{e\in\Gamma_h}C\|\mathbf{u}_t\|_{-\frac{1}{2},e}\|\phi\|_{\frac{1}{2},e}\leq\sum\limits_{e\in\Gamma_h}C\|\mathbf{u}\|_{\frac{1}{2},e} \|\phi\|_{\frac{1}{2},e}\leq\sum\limits_{e\in\Gamma_h}C\|\mathbf{u}\|_{0,e}\|\phi\|_{0,e},
   \end{eqnarray}
   and
   \begin{eqnarray}\label{eq:C4}
     \left|\sum\limits_{e\in\Gamma_h^R} \int\limits_{e}\varphi_t\phi dS\right|\leq \sum\limits_{e\in\Gamma_h^R} C \|\varphi_t\|_{-\frac{1}{2},e}\|\phi\|_{\frac{1}{2},e}\leq \sum\limits_{e\in\Gamma_h^R} C \|\varphi\|_{\frac{1}{2},e}\|\phi\|_{\frac{1}{2},e}\leq \sum\limits_{e\in\Gamma_h^R} C \|\varphi\|_{0,e}\|\phi\|_{0,e},
   \end{eqnarray}
   where all constants $C>0$ are independent of $\mathbf{u},\varphi,\mathbf{v},$ and $\phi$.  Set $C_{\text{cont}}=\max\left\{C,\frac{\lambda+8\mu}{\rho_1}\right\}$, then combining~\eqref{eq:C1}-\eqref{eq:C5} together and applying the arithmetic geometric average inequality, we obtain~\eqref{eq:C0}. \qed

   \subsection{\textit{A priori} error estimate}
   \par{}In this section, we will give \textit{a priori} error estimates in the energy norm for the IPDG method.  First of all, we will derive an error equation.  For $(\mathbf{u},\varphi)\in \mathbf{H}^{1+m}\times H^{1+m}$ with $m>1/2$ and $(\mathbf{v},\phi)\in \mathbf{V}_1(h)\times V_2(h)$, we define
   \begin{equation}\label{error1}
     r_h(\mathbf{u},\varphi;\mathbf{v},\phi)=\sum\limits_{e\in\Gamma_h\cup\Gamma_{I}^1}\int_e[\mathbf{v}]\cdot\{\sigma(\mathbf{u}) -\Pi_h(\sigma(\mathbf{u}))\}\,dS+\sum\limits_{e\in\Gamma_h\cup\Gamma_{I}^2}\int_e[\phi]\cdot\{\nabla\varphi-\Pi_h(\nabla\varphi)\}\,dS.
   \end{equation}
   Here $\Pi_h$ denotes the $L^2$-projection onto the associated finite element space ($\mathbf{V}_1(h)$ or $V_2(h)$).  The assumption $(\mathbf{u},\varphi)\in \mathbf{H}^{1+m}\times H^{1+m}$ ensures that $r_h(\mathbf{u},\varphi;\mathbf{v},\phi)$ is well defined. From the definition \eqref{error1} it is obvious that $r_h(\mathbf{u},\varphi;\mathbf{v},\phi)=0$ when $(\mathbf{v},\phi)\in \mathbf{H}^{1+m}\times H^{1+m}$.
   \begin{lemma}\label{lemma3.2}
   {Assume the analytical solution $(\mathbf{u},\varphi)$ of~\eqref{eq:rdfsi1}-\eqref{eq:rdfsi6} satisfy
   \[
     \mathbf{u}\in L^{\infty}\left(J;\mathbf{H}^{1+m}(\Omega)\right),\,\,\mathbf{u}_{tt}\in L^{1}\left(J;\mathbf{H}^m(\Omega)\right),
   \]
   and
   \[
     \varphi\in L^{\infty}\left(J;H^{1+m}(\Omega_R)\right),\,\,\varphi_{tt}\in L^{1}\left(J;H^m(\Omega_R)\right).
   \]
   Let $(\mathbf{u}^h$, $\varphi^h)$ be the semidiscrete DG approximation solution. Then, the error $(\mathbf{e}^u,e^\varphi)=(\mathbf{u}-\mathbf{u}^h,\varphi-\varphi^h)$ satisfies:
   \begin{eqnarray}\label{error2}
     \frac{\rho_2}{\rho_1}\left\langle\mathbf{e}^u_{tt},\mathbf{v}\right\rangle+\frac{1}{c^2}\left\langle e^\varphi_{tt},\phi\right\rangle+a_h(\mathbf{e}^u,e^\varphi;\mathbf{v},\phi)=r_h(\mathbf{e}^u,e^\varphi;\mathbf{v},\phi),
   \end{eqnarray}
   for $(\mathbf{v},\phi)\in \mathbf{\mathcal{D}}_k(\mathbf{\mathcal{E}}_h)\times\mathcal{D}_k(\mathcal{K}_h),$  and $r_h$ is given in~\eqref{error1}.
   }
   \end{lemma}
   {\bf Proof:} Let $(\mathbf{v},\phi)\in \mathbf{\mathcal{D}}_k(\mathbf{\mathcal{E}}_h)\times\mathcal{D}_k(\mathcal{K}_h)$. Since $\mathbf{u}_{tt}\in L^1\left(J;\mathbf{L}^2(\Omega)\right)$, $\varphi_{tt}\in L^1\left(J;L^2(\Omega_R)\right)$, we have $\langle\mathbf{u}_{tt},\mathbf{v}\rangle=(\mathbf{u}_{tt},\mathbf{v}),~\langle\varphi_{tt},\phi\rangle=(\varphi_{tt},\phi)$ almost everywhere in $[0,T]$.  Hence, using the discrete formulation in~\eqref{eq:IPDGVF1}, we obtain that
   \begin{align*}
     \frac{\rho_2}{\rho_1}\left(\mathbf{e}^u_{tt},\mathbf{v}\right)&+\frac{1}{c^2}\left(e^\varphi_{tt},\phi\right)+a_h(\mathbf{e}^u,e^\varphi;\mathbf{v},\phi)=\\
     &\quad\frac{\rho_2}{\rho_1}\left(\mathbf{u}_{tt},\mathbf{v}\right)+\frac{1}{c^2}\left(\varphi_{tt},\phi\right)+a_h(\mathbf{u},\varphi;\mathbf{v},\phi)-L_h(\mathbf{v},\phi)\quad\text{a.e.\ in~} [0,T].
   \end{align*}
   By the definition of $a_h$, the fact that $[\mathbf{u}]=0$, $[\varphi]=0$ on all edges and the properties of the $L^2$-projection $\Pi_h$, we obtain
   \begin{eqnarray*}
   a_h(\mathbf{u},\varphi;\mathbf{v},\phi)&=&\sum\limits_{E\in
   \mathbf{\mathcal{E}}_h}\left(\frac{\lambda}{\rho_1}(\nabla\cdot\mathbf{u},\nabla\cdot \mathbf{v})_{0,E}+\frac{2\mu}{\rho_1}\left(\varepsilon(\mathbf{u}):\varepsilon(\mathbf{v})\right)_{F(E)}\right)+\sum\limits_{K\in\mathcal{K}_h}(\nabla\varphi,\nabla\phi)_{0,K}\\
   &&-\sum\limits_{e\in \Gamma_I^1}\int_e\{\Pi_h(\sigma(\mathbf{u}))\cdot\mathbf{n}\}[\mathbf{v}]dS-\sum\limits_{e\in\Gamma_I^2}
   \int_e\{\Pi_h(\nabla\varphi)\mathbf{n}\}[\phi]\,dS+ \sum\limits_{e\in\Gamma_h}\int_e\varphi_t\mathbf{n}\mathbf{v}\,dS\\
   &&-\sum\limits_{e\in\Gamma_h}\int_e\mathbf{u}_t\cdot\mathbf{n}\phi\, dS+\sum\limits_{e\in\Gamma_h^R} \int_{e}\varphi_t\phi\,dS.
   \end{eqnarray*}

   Since $\mathbf{u}_{tt}\in L^1\left(J;\mathbf{L}^2(\Omega)\right)$, $\varphi_{tt}\in L^1\left(J;\mathbf{L}^2(\Omega_R)\right)$ and $\varphi^i\in L^1\left(J;L^2(\Omega_R)\right)$, we have that $\text{div}(\sigma(\mathbf{u}))\in L^2(\Omega)$ and $\Delta\varphi\in L^2(\Omega_R)$ almost everywhere in $[0,T]$, which means that $\text{div}(\sigma(\mathbf{u}))$ and $\Delta\varphi$ have continuous normal components across all interior faces.  Therefore, using integration by parts elementwisely and combining with the trace operators yield
   \begin{eqnarray*}
     a_h(\mathbf{u},\varphi;\mathbf{v},\phi)&=&-\sum\limits_{E\in\mathbf{\mathcal{E}}_h}\frac{1}{\rho_1}(\text{div}(\sigma(\mathbf{u})),\mathbf{v})_{0,E}+\sum\limits_{e\in \Gamma_I^1}\int_e\{\sigma(\mathbf{u})\cdot\mathbf{n}\}[\mathbf{v}]\,dS\\
   &&-\sum\limits_{e\in \Gamma_I^1}\int_e\{\Pi_h(\sigma(\mathbf{u}))\cdot\mathbf{n}\}[\mathbf{v}]dS
   -\sum\limits_{K\in\mathcal{K}_h}(\Delta\varphi,\phi)_{0,K}+\sum\limits_{e\in\Gamma_I^2}\int_e\{\nabla\varphi\mathbf{n}\}[\phi]\,dS\\
 &&-\sum\limits_{e\in\Gamma_I^2}\int_e\{\Pi_h(\nabla\varphi)\mathbf{n}\}[\phi]\,dS + L_h(\mathbf{v},\phi).
   \end{eqnarray*}
   From the definition of $r_h(\mathbf{u},\varphi;\mathbf{v},\phi)$ in~\eqref{error1}, we conclude that
   \begin{eqnarray*}
     &&\frac{\rho_2}{\rho_1}\left(\mathbf{u}_{tt},\mathbf{v}\right)+\frac{1}{c^2}\left(\varphi_{tt},\phi\right)+a_h(\mathbf{u},\varphi;\mathbf{v},\phi)\\
     &=&\left(\frac{\rho_2}{\rho_1}\mathbf{u}_{tt}-\frac{1}{\rho_1}\text{div}(\sigma(\mathbf{u})),\mathbf{v}\right)
     +\left(\frac{1}{c^2}\varphi_{tt}-\Delta\varphi,\phi\right) + L_h(\mathbf{v},\phi) + r_h(\mathbf{u},\varphi;\mathbf{v},\phi),\\
     &=&L_h(\mathbf{v},\phi) + r_h(\mathbf{u},\varphi;\mathbf{v},\phi),
   \end{eqnarray*}
   and obtain
   \begin{eqnarray*}
     &&\frac{\rho_2}{\rho_1}\left(\mathbf{e}^u_{tt},\mathbf{v}\right)+\frac{1}{c^2}\left( e^\varphi_{tt},\phi\right)+a_h(\mathbf{e}^u,e^\varphi;\mathbf{v},\phi) \\
     &=&L_h(\mathbf{v},\phi)-L_h(\mathbf{v},\phi)+r_h(\mathbf{u},\varphi;\mathbf{v},\phi) = r_h(\mathbf{u},\varphi;\mathbf{v},\phi),
   \end{eqnarray*}
   where we have used the differential equations~\eqref{eq:rdfsi1} and~\eqref{eq:rdfsi2}. \qed

   \paragraph{}Next, we recall some approximation properties, see~\cite{Ciarlet:1978} for more detail.
   \begin{lemma}\label{lemma3.3}
     {Let $E\in\mathbf{\mathcal{E}}_h$, $K\in\mathcal{K}_h$. Then the following hold:\\
       (i) For $(\mathbf{v},\phi)\in \mathbf{H}^{m}(E)\times H^{m}(K)$, $m\geq0$, we have
   \begin{equation*}
     \|\mathbf{v}\|_{0,E}+\|\phi\|_{0,K}\leq Ch_{E}^{\min\{m,k+1\}}\|\mathbf{v}\|_{m,E}+Ch_{K}^{\min\{m,k+1\}}\|\phi\|_{m,K}.
   \end{equation*}
   (ii) For $(\mathbf{v},\phi)\in \mathbf{H}^{1+m}(E)\times H^{1+m}(K)$, $m>\frac{1}{2}$, we have
   \begin{align*}
     \|\nabla\mathbf{v}-\nabla(\Pi_h\mathbf{v})\|_{0,E}+\|\nabla\phi-\nabla(\Pi_h\phi)\|_{0,K} & \leq Ch_{E}^{\min\{m,k\}}\|\mathbf{v}\|_{1+m,E}+Ch_{K}^{\min\{m,k\}}\|\phi\|_{1+m,K},\\
     \|\mathbf{v}-\Pi_h\mathbf{v}\|_{0,\partial E}+\|\phi-\Pi_h\phi\|_{0,\partial K} & \leq Ch_{E}^{\min\{m,k+1\}-\frac{1}{2}}\|\mathbf{v}\|_{1+m,E}+Ch_{K}^{\min\{m,k+1\}-\frac{1}{2}}\|\phi\|_{1+m,K},\\
     \|\nabla\mathbf{v}-\Pi_h(\nabla\mathbf{v})\|_{0,\partial E}+\|\nabla\phi-\Pi_h(\nabla\phi)\|_{0,\partial K} &\leq Ch_{E}^{\min\{m,k\}-\frac{1}{2}}\|\mathbf{v}\|_{1+m,E} + Ch_{K}^{\min\{m,k\}-\frac{1}{2}}\|\phi\|_{1+m,K},
   \end{align*}
   where $C$ is independent of local mesh size $h$.
   }
   \end{lemma}

   As a consequence of the approximation properties in Lemma~\ref{lemma3.3}, we obtain the following results.
   \begin{lemma}\label{lemma3.4}
   {Let $(\mathbf{v},\phi)\in H^{1+m}(E)\times H^{1+m}(K)$, $m>\frac{1}{2}$.  Then the following hold:\\
   (i) We have
   \begin{equation*}
     \|\mathbf{u}-\Pi_h\mathbf{u}\|_{h}+\|\varphi-\Pi_h\varphi\|_{h}\leq C_Ah^{\min\{m, k\}}\left(\|\mathbf{u}\|_{1+m,\Omega}+\|\phi\|_{1+m,\Omega_R}\right),
   \end{equation*}
   with a constant $C_A$ that is independent of the mesh size. \\
   (ii) For $(\mathbf{v},\phi)\in \mathbf{V}_1(h)\times V_2(h)$, the form $r_h(\mathbf{u},\varphi;\mathbf{v},\phi)$ in~\eqref{error1} can be bounded by
   \begin{eqnarray*}
     |r_h(\mathbf{u},\varphi;\mathbf{v},\phi)| \leq C_Rh^{\min\{m,k\}}\left(\left(\sum\limits_{E\in\mathcal{E}_h}\|[\mathbf{v}]\|_{0,\partial E}^2\right)^{\frac{1}{2}}\|\mathbf{u}\|_{1+m,\Omega}
     +\left(\sum\limits_{K\in\mathcal{K}_h}\|[\phi]\|^2_{0,\partial K}\right)^{\frac{1}{2}}\|\varphi\|_{1+m,\Omega_R}\right),
   \end{eqnarray*}
   with a constant $C_R$ that is independent of $h$ and depends only $\alpha$ and the constants in Lemma~\ref{lemma3.3}.
   }
      \end{lemma}
   \begin{lemma}\label{lemma3.5}
   {Assume the analytical solution $(\mathbf{u},\varphi)$ of~\eqref{eq:rdfsi1}-\eqref{eq:rdfsi6} satisfy
   \[
     \mathbf{u}\in L^{\infty}\left(J;\mathbf{H}^{1+m}(\mathbf{\mathcal{E}}_h)\right),\,\,\mathbf{u}_t\in L^{\infty}\left(J;\mathbf{H}^{1+m}(\mathbf{\mathcal{E}}_h)\right),
   \]
   and
   \[
     \varphi\in L^{\infty}\left(J;H^{1+m}(\mathcal{K}_h)\right),\,\,\varphi_t\in L^{\infty}\left(J; H^{1+m}(\mathcal{K}_h)\right)
   \]
   for a regularity exponent $m > 1/2$.  Let
   \[
     \mathbf{v}\in \mathcal{C}^0\left(J;\mathbf{V}_1(h)\right),\,\,\mathbf{v}_t\in L^1\left(J;\mathbf{V}_1(h)\right),\,
     \phi\in \mathcal{C}^0\left(J;V_2(h)\right),\,\,\phi_t\in L^1\left(J;V_2(h)\right).
   \]
   Then we have
   \begin{align*}
       & \int_0^T|r_h(\mathbf{u},\varphi;\mathbf{v}_t,\phi_t)|dt \\
       \leq&\, C_Rh^{\min\{m,k\}}\Big\{\|\mathbf{v}\|_{L^{\infty}(J;\mathbf{V}_1(h))}
         \left(2\|\mathbf{u}\|_{L^{\infty}(J;\mathbf{H}^{1+m}(\Omega))}+T\|\mathbf{u}_t\|_{L^{\infty}(J;\mathbf{H}^{1+,}(\Omega))}\right)\\
       &\qquad+\|\phi\|_{L^{\infty}(J;V_2(h))}\left(2\|\varphi\|_{L^{\infty}(J;H^{1+m}(\Omega_R))}+T\|\varphi_t\|_{L^{\infty}(J;H^{1+m}(\Omega_R))}\right)\Big\},
   \end{align*}
   where $C_R$ is the constant from the bound (ii) in Lemma~\ref{lemma3.4}.
   }
   \end{lemma}
   {\bf Proof:} From the definition of $r_h$ in~\eqref{error1} and integration by parts, we obtain
   \begin{align*}
     &\, \int_0^T|r_h(\mathbf{u},\varphi;\mathbf{v}_t,\phi_t)|\,dt \\
    =& \,\int_0^T \sum\limits_{e\in\Gamma_h\cup\Gamma_{I}^1}\int_e[\mathbf{v}_t]\cdot\{\sigma(\mathbf{u}) -\Pi_h(\sigma(\mathbf{u}))\} \,dSdt + \int_0^T\sum\limits_{e\in\Gamma_h\cup\Gamma_{I}^2}\int_e[\phi_t]\cdot\{\nabla\varphi-\Pi_h(\nabla\varphi)\}\,dSdt\\
    =&\,-\int_0^T \sum\limits_{e\in\Gamma_h\cup\Gamma_{I}^1}\int_e[\mathbf{v}]\cdot\{\sigma(\mathbf{u}_t)-\Pi_h(\sigma(\mathbf{u}_t))\}\,dSdt -  \int_0^T\sum\limits_{e\in\Gamma_h\cup\Gamma_{I}^2}\int_e[\phi]\cdot\{\nabla\varphi_t-\Pi_h(\nabla\varphi_t)\}\,dSdt\\
    &\qquad+\left. \sum\limits_{e\in\Gamma_h\cup\Gamma_{I}^1}\int_e[\mathbf{v}]\cdot\{\sigma(\mathbf{u}) -\Pi_h(\sigma(\mathbf{u}))\}\,dS\right|_{t=0}^{t=T}
     +\left.\sum\limits_{e\in\Gamma_h\cup\Gamma_{I}^2}\int_e[\phi]\cdot\{\nabla\varphi-\Pi_h(\nabla\varphi)\}\,dS\right|_{t=0}^{t=T}\\
       =&\,-\int_0^T|r_h(\mathbf{u}_t,\varphi_t;\mathbf{v},\phi)|\,dt+r_h(\mathbf{u},\varphi;\mathbf{v},\phi)\Big|_{t=0}^{t=T}.
   \end{align*}
   Lemma~\ref{lemma3.4} implies
    \begin{align*}
      \int_0^T|r_h(\mathbf{u}_t,\varphi_t;\mathbf{v},\phi)|\,dt \leq C_Rh^{\min\{m,k\}}T &\left(\|\mathbf{v}\|_{L^{\infty}(J;\mathbf{V}_1(h))}\|\mathbf{u}_t\|_{L^{\infty}(J;\mathbf{H}^{1+m}(\Omega))}\right. \\
      &\qquad \left.+\|\phi\|_{L^{\infty}(J;V_2(h))}\|\varphi_t\|_{L^{\infty}(J;H^{1+m}(\Omega_R))}\right),
   \end{align*}
   and
   \begin{align*}
     \left|r_h(\mathbf{u},\varphi;\mathbf{v},\phi)\big|_{t=0}^{t=T}\right| \leq 2C_Rh^{\min\{m,k\}}&\left(\|\mathbf{v}\|_{L^{\infty}(J;\mathbf{V}_1(h))}\|\mathbf{u}\|_{L^{\infty}(J;\mathbf{H}^{1+m}(\Omega))}\right. \\
     &\qquad\left.+\|\phi\|_{L^{\infty}(J;V_2(h))}\|\varphi\|_{L^{\infty}(J;H^{1+m}(\Omega_R))}\right),
   \end{align*}
   which finishes the proof of the lemma.\qed

   \begin{theorem}\label{theorem3.1}
   {Assume that the solution of \eqref{eq:rdfsi1}-\eqref{eq:rdfsi6} satisfy
   \[
     \mathbf{u}\in L^{\infty}\left(J;\mathbf{H}^{1+m}(\mathbf{\mathcal{E}}_h)\right),\,\mathbf{u}_t\in L^{\infty}\left(J; \mathbf{H}^{1+m}(\mathbf{\mathcal{E}}_h)\right),\,\mathbf{u}_{tt}\in L^{1}\left(J;\mathbf{H}^m(\mathbf{\mathcal{E}}_h)\right),
    \]
    and
    \[
      \varphi\in L^{\infty}\left(J;H^{1+m}(\mathcal{K}_h)\right),\,\varphi_t\in L^{\infty}\left(J;H^{1+m}(\mathcal{K}_h)\right),\,\varphi_{tt}\in L^{1}\left(J;H^m(\mathcal{K}_h)\right),
    \]
   for a regularity exponent $m>1/2$, and let $\mathbf{u}^h$, $\varphi^h$ be the semidiscrete discontinuous Galerkin approximation obtained by~\eqref{eq:IPDGVF1}. Then the error $\mathbf{e}^u=\mathbf{u}-\mathbf{u}^h$ and $e^\varphi=\varphi-\varphi^h$ satisfy the estimate:
   \begin{align*}
     & \|\mathbf{e}^u_t\|_{L^{\infty}(J;L^2(\Omega))} + \|\mathbf{e}^u\|_{L^{\infty}(J;\mathbf{V}_1)} \leq C\left(\|\mathbf{e}^u_t(\mathbf{0})\|_{0,\Omega} + \|\mathbf{e}^u(\mathbf{0})\|_h\right) \\
       & \qquad\qquad + Ch^{\min\{m,k\}}\left(\|\mathbf{u}\|_{L^{\infty}(J;\mathbf{H}^{1+m}(\Omega))} + T\|\mathbf{u}_t\|_{L^{\infty}(J; \mathbf{H}^{1+m}(\Omega))}+\|\mathbf{u}_{tt}\|_{L^{1}(J;\mathbf{H}^m(\Omega))}\right),
   \end{align*}
   and
   \begin{align*}
    & \|e^\varphi_t\|_{L^{\infty}(J;L^2(\Omega_R))}+ \|e^\varphi\|_{L^{\infty}(J;V_2(h))} \leq C\left( \|e^\varphi_t(0)\|_{0,\Omega} + \|e^\varphi(0)\|_h\right) \\
   &\qquad\qquad + Ch^{\min\{m,k\}}\left(\|\varphi\|_{L^{\infty}(J;H^{1+m}(\Omega_R))}
   + T\|\varphi_t\|_{L^{\infty}(J; H^{1+m}(\Omega_R))}+\|\varphi_{tt}\|_{L^{1}(J;H^m(\Omega_R))}\right),
   \end{align*}
   with the constant $C$ that is independent of $h$ and $T$.}
   \end{theorem}
   {\bf Proof:} Because of Theorem~\ref{thm31}, we have
   \[
     \mathbf{e}^u\in \mathcal{C}^0\left(T;\mathbf{V}_1(h)\right)\cap \mathcal{C}^1\left(J;\mathbf{L}^2(\Omega)\right),\quad e^{\varphi}\in \mathcal{C}^0\left(J;V_2(h)\right)\cap \mathcal{C}^1\left(J;L^2(\Omega_R)\right).
   \]
   Using the symmetry of $a_h$ and the error equation in Lemma~\ref{lemma3.2}, we obtain
   \begin{align}
   \nonumber
   &\, \frac{1}{2}\frac{d}{dt}\left(\|\mathbf{e}_t^u\|_{0,\Omega}^2+\|e_t^{\varphi}\|_{0,\Omega_R}^2
     +a_h(\mathbf{e}^u,e^{\varphi};\mathbf{e}^u,e^{\varphi})\right) \\
     \nonumber
    =&\, (\mathbf{e}_{tt}^u,\mathbf{e}_t^u)+(e_{tt}^{\varphi},e_t^{\varphi})+a_h(\mathbf{e}^u,e^{\varphi};\mathbf{e}^u_t,e^{\varphi}_t)\\
   \nonumber
     =&\,\left(\mathbf{e}_{tt}^u,(\mathbf{u}-\Pi_h\mathbf{u})_t\right)+\left(e_{tt}^{\varphi},(\varphi-\Pi_h\varphi)_t\right) + a_h(\mathbf{e}^u,e^{\varphi};(\mathbf{u}-\Pi_h\mathbf{u})_t,(\varphi-\Pi_h\varphi)_t) \\
     &\qquad\qquad +r_h\left(\mathbf{u},\varphi;(\Pi_h\mathbf{u}-\mathbf{u}^h)_t,(\Pi_h\varphi-\varphi^h)_t\right) \label{eq:thm4.1}.
   \end{align}
   We fix $\tau\in[0,T]$ and integrate~\eqref{eq:thm4.1} over the time interval $(0,\tau)$. This yields
   \begin{align*}
     &\,\frac{1}{2}\|\mathbf{e}_t^u(\tau)\|_{0,\Omega}^2 + \frac{1}{2}\|e_t^{\varphi}(\tau)\|_{0,\Omega_R}^2+\frac{1}{2}a_h\left(\mathbf{e}^u(\tau),e^{\varphi}(\tau);\mathbf{e}^u(\tau),e^{\varphi}(\tau)\right) \\
     = &\,\frac{1}{2}\|\mathbf{e}_t^u(\mathbf{0})\|_{0,\Omega}^2 +\frac{1}{2}\|e_t^{\varphi}(0)\|_{0,\Omega_R}^2+\frac{1}{2}a_h(\mathbf{e}^u(0),e^{\varphi}(0);\mathbf{e}^u(\mathbf{0}),e^{\varphi}(0)) \\
     &\qquad+ \int_0^\tau a_h(\mathbf{e}^u,e^{\varphi};(\mathbf{u}-\Pi_h\mathbf{u})_t,(\varphi-\Pi_h\varphi)_t)\,dt + \int_0^{\tau}(\mathbf{e}_{tt}^u,(\mathbf{u}-\Pi_h\mathbf{u})_t)\,dt \\
     &\qquad\quad +\int_0^{\tau}(e_{tt}^{\varphi},(\varphi-\Pi_h\varphi)_t)\,dt + \int_0^{\tau}r_h\left(\mathbf{u},\varphi;(\Pi_h\mathbf{u}-\mathbf{u}^h)_t,(\Pi_h\varphi-\varphi^h)_t\right)\,dt.
   \end{align*}
   Integration by parts of the fifth and sixth terms on the right-hand side yields
   \[
     \int_0^{\tau}\left(\mathbf{e}_{tt}^u,(\mathbf{u}-\Pi_h\mathbf{u})_t\right)\,dt = \int_0^{\tau}\left(\mathbf{e}_{t}^u,(\mathbf{u}-\Pi_h\mathbf{u}\right)_{tt})\,dt + \left(\mathbf{e}_{t}^u,(\mathbf{u}-\Pi_h\mathbf{u})_t\right)\Big|_{t=0}^{t=\tau},
 \]
 and
 \[
   \int_0^{\tau}\left(e_{tt}^{\varphi},(\varphi-\Pi_h\varphi)_t\right)\,dt = \int_0^{\tau}\left(\mathbf{e}_{t}^{\varphi},(\varphi-\Pi_h\varphi)_{tt}\right)\,dt+\left(\mathbf{e}_{t}^{\varphi},(\varphi-\Pi_h\varphi)_t\right)\Big|_{t=0}^{t=\tau}.
 \]
 From the stability properties of $a_h$ in Lemma~\ref{lemma4.1} and Lemma~\ref{lemma4.2} and the H\"{o}lder's inequalities, we conclude that
   \begin{align*}
     &\,\frac{1}{2}\|\mathbf{e}_t^u(\tau)\|_{0,\Omega}^2 + \frac{1}{2}\|e_t^{\varphi}(\tau)\|_{0,\Omega_R}^2+\frac{1}{2}C_{\text{coer}}\|\mathbf{e}^u(\tau)\|^2_h +\frac{1}{2}C_{\text{coer}}\|e^{\varphi}(\tau)\|^2_h\\
     \leq&\,\frac{1}{2}\|\mathbf{e}_t^u(0)\|_{0,\Omega}^2 +\frac{1}{2}\|e_t^{\varphi}(0)\|_{0,\Omega_R}^2+\frac{1}{2}C_{\text{cont}}\|\mathbf{e}^u(0)\|^2_h+\frac{1}{2}C_{\text{cont}}\|e^{\varphi}(0)\|^2_h\\
     &\,\,+\|\mathbf{e}^u_t\|_{L^{\infty}(0,T;L^2(\Omega))} \left(\|(\mathbf{u}-\Pi_h\mathbf{u})_{tt}\|_{L^{1}(0,T;L^2(\Omega))}+2\|(\mathbf{u}-\Pi_h\mathbf{u})_{t}\|_{L^{\infty}(0,T;L^2(\Omega))}\right)\\
   &\quad+\|e^{\varphi}_t\|_{L^{\infty}(0,T;L^2(\Omega_R))} \left(\|(\varphi-\Pi_h\varphi)_{tt}\|_{L^{1}(0,T;L^2(\Omega_R))}+2\|(\varphi-\Pi_h\varphi)_t\|_{L^{\infty}(0,T;L^2(\Omega_R))}\right)\\
   &\qquad+C_{\text{cont}}T\left(\|\mathbf{e}^u\|_{L^{\infty}(J;\mathbf{V}_1(h))}\|(\mathbf{u}-\Pi_h\mathbf{u})_{t}\|_{L^{\infty}(J;\mathbf{V}_1(h))} + \|e^{\varphi}\|_{L^{\infty}(J;V_2(h))} \|(\varphi-\Pi_h\varphi)_t\|_{L^{\infty}(J;V_2(h))}\right)\\
   &\qquad\quad+\left|\int_0^{\tau}r_h(\mathbf{u},\varphi;(\Pi_h\mathbf{u}-\mathbf{u}^h)_t,(\Pi_h\varphi-\varphi^h)_t)\,dt\right|.
   \end{align*}
   Since this inequality holds for every $\tau \in J$, it also holds for the maximum over $J$, that is
   \begin{align*}
     &\,\|\mathbf{e}_t^u(\tau)\|_{L^{\infty}(J;\mathbf{L}^2(\Omega))}^2 + \|e_t^{\varphi}(\tau)\|_{L^{\infty}(J;L^2(\Omega_R))} ^2+C_{\text{coer}}\|\mathbf{e}^u(\tau)\|^2_{L^{\infty}(J;\mathbf{L}^2(\Omega))}
     + C_{\text{coer}}\|e^{\varphi}(\tau)\|^2_{L^{\infty}(J;L^2(\Omega_R))} \\
     \leq&\,\|\mathbf{e}_t^u(\mathbf{0})\|_{0,\Omega}^2+\|e_t^{\varphi}(0)\|_{0,\Omega_R}^2
     +C_{\text{cont}}\|\mathbf{e}^u(\mathbf{0})\|^2_h+C_{\text{cont}}\|e^{\varphi}(0)\|^2_h+T_1+T_2+T_3+T_4+T_5,
   \end{align*}
   with
   \begin{align*}
     T_1 &= 2\|\mathbf{e}^u_t\|_{L^{\infty}(J;\mathbf{L}^2(\Omega))} \left(\|(\mathbf{u}-\Pi_h\mathbf{u})_{tt}\|_{L^{1}(J;\mathbf{L}^2(\Omega))}+2\|(\mathbf{u}-\Pi_h\mathbf{u})_{t}\|_{L^{\infty}(J;\mathbf{L}^2(\Omega))}\right), \\
     T_2 &= 2\|e^{\varphi}_t\|_{L^{\infty}(J;L^2(\Omega_R))} \left(\|(\varphi-\Pi_h\varphi)_{tt}\|_{L^{1}(J;L^2(\Omega_R))}+2\|(\varphi-\Pi_h\varphi)_t\|_{L^{\infty}(J;L^2(\Omega_R))}\right), \\
     T_3 &= 2C_{\text{cont}}T\|\mathbf{e}^u\|_{L^{\infty}(J;\mathbf{V}_1(h))}\|(\mathbf{u}-\Pi_h\mathbf{u})_{t}\|_{L^{\infty}(J;\mathbf{V}_1(h))}, \\
     T_4 &= 2C_{\text{cont}}T\|e^{\varphi}\|_{L^{\infty}(J;V_2(h))} \|(\varphi-\Pi_h\varphi)_t\|_{L^{\infty}(J;V_2(h))}, \\
     T_5 &= 2\left|\int_0^{\tau}r_h(\mathbf{u},\varphi;(\Pi_h\mathbf{u}-\mathbf{u}^h)_t,(\Pi_h\varphi-\varphi^h)_t)dt\right|.
   \end{align*}
   Using the inequality that $|ab|\leq\frac{1}{2\epsilon}a^2+\frac{\epsilon}{2}b^2$, and the approximation results in Lemma~\ref{lemma3.3}, we conclude that
   \begin{align*}
     T_1 &\leq \frac{1}{2}\|\mathbf{e}^u_t\|^2_{L^{\infty}(J;\mathbf{L}^2(\Omega))}+ 2\left(\|(\mathbf{u}-\Pi_h\mathbf{u})_{tt}\|_{L^{1}(J;\mathbf{L}^2(\Omega))}+2\|(\mathbf{u}-\Pi_h\mathbf{u})_{t}\|_{L^{\infty}(J;\mathbf{L}^2(\Omega))}\right)^2 \\
     &\leq \frac{1}{2}\|\mathbf{e}^u_t\|^2_{L^{\infty}(J;\mathbf{L}^2(\Omega))}+ 4\|(\mathbf{u}-\Pi_h\mathbf{u})_{tt}\|^2_{L^{1}(J;\mathbf{L}^2(\Omega))}+16\|(\mathbf{u}-\Pi_h\mathbf{u})_{t}\|_{L^{\infty}(J;\mathbf{L}^2(\Omega))}^2 \\
     &\leq \frac{1}{2}\|\mathbf{e}^u_t\|^2_{L^{\infty}(J;\mathbf{L}^2(\Omega))}+ Ch^{2\min\{m,k\}}\left(\|\mathbf{u}_{tt}\|^2_{L^{1}(J;\mathbf{H}^{m}(\Omega))}+h^2\|\mathbf{u}_{t}\|_{L^{\infty}(J;\mathbf{H}^{1+m}(\Omega))}^2\right)
   \end{align*}
   and
   \begin{align*}
     T_2 &\leq \frac{1}{2}\|e^{\varphi}_t\|^2_{L^{\infty}(J;L^2(\Omega_R))}+ 2\left(\|(\varphi-\Pi_h\varphi)_{tt}\|_{L^{1}(J;L^2(\Omega_R))}+2\|(\varphi-\Pi_h\varphi)_t\|_{L^{\infty}(J;L^2(\Omega_R))}\right)^2 \\
      &\leq \frac{1}{2}\|e^{\varphi}_t\|^2_{L^{\infty}(J;L^2(\Omega_R))}+ 4\|(\varphi-\Pi_h\varphi)_{tt}\|^2_{L^{1}(J;L^2(\Omega_R))}+16\|(\varphi-\Pi_h\varphi)_t\|^2_{L^{\infty}(J;L^2(\Omega_R))} \\
      &\leq \frac{1}{2}\|e^{\varphi}_t\|^2_{L^{\infty}(J;L^2(\Omega_R))}+ Ch^{2\min\{m,k\}}\left(\|\varphi_{tt}\|^2_{L^{1}(J;H^{m}(\Omega_R))}+h^2\|\varphi_{t}\|_{L^{\infty}(J;H^{1+m}(\Omega_R))}^2\right)
   \end{align*}
   with the constant $C$ depends only on the constants in Lemma~\ref{lemma3.3}. Similarly,
   \begin{align*}
     T_3 &\leq \frac{1}{4}C_{\text{coer}}\|\mathbf{e}^u\|^2_{L^{\infty}(J;\mathbf{V}_1(h))}+4\frac{C_{\text{cont}}^2}{C_{\text{coer}}}T^2\|(\mathbf{u} -\Pi_h\mathbf{u})_{t}\|^2_{L^{\infty}(J;\mathbf{V}_1(h))} \\
     &\leq \frac{1}{4}C_{\text{coer}}\|\mathbf{e}^u\|^2_{L^{\infty}(J;\mathbf{V}_1(h))}+T^2Ch^{2\min\{m,k\}}\|\mathbf{u}_t\|^2_{L^{\infty}(J;\mathbf{H}^{1+,}(\Omega))}
   \end{align*}
   and
   \begin{align*}
     T_4 &\leq \frac{1}{4}C_{\text{coer}}\|e^{\varphi}\|^2_{L^{\infty}(J;V_2(h))}+4\frac{C_{\text{cont}}^2}{C_{\text{coer}}}T^2\|(\varphi -\Pi_h\varphi)_{t}\|^2_{L^{\infty}(J;V_2(h))} \\
     &\leq \frac{1}{4}C_{\text{coer}}\|e^{\varphi}\|^2_{L^{\infty}(J;V_2(h))}+T^2Ch^{2\min\{m,k\}}\|\varphi_t\|^2_{L^{\infty}(J;H^{1+m}(\Omega_R))}
   \end{align*}
   with the constant $C$ depends only on $C_{\text{cont}}$, $C_{\text{coer}}$ and the constants $C_A$ in Lemma~\ref{lemma3.4}.
    \par{} To bound the term $T_5$, we use the Lemma~\ref{lemma3.5} to obtain
    \[
      T_5\leq2C_Rh^{\min\{m,k\}}\left(\mathcal{R}_1\|\Pi_h\mathbf{u}-\mathbf{u}^h\|_{L^{\infty}(J;\mathbf{V}_1(h))}
      +\mathcal{R}_2\|\Pi_h\varphi-\varphi^h\|_{L^{\infty}(J;V_2(h))}\right)
    \]
    with
    \begin{align*}
      \mathcal{R}_1 &:= \left(T\|\mathbf{u}_{t}\|_{L^{\infty}(J;\mathbf{H}^{1+m}(\Omega))}+2\|\mathbf{u}\|_{L^{\infty}(J;\mathbf{H}^{1+m}(\Omega))}\right), \\
      \mathcal{R}_2 &:= \left(T\|\varphi_{t}\|_{L^{\infty}(J;H^{1+m}(\Omega_R))}+2\|\varphi\|_{L^{\infty}(J;H^{1+m}(\Omega_R))}\right).
    \end{align*}
    The triangle inequality, the geometric-arithmetic mean, and the approximation properties of $\Pi_h$ in~Lemma~\ref{lemma3.4} then yield
    \begin{align*}
      T_5\leq&\, 2C_Rh^{\min\{m,k\}}\left(\mathcal{R}_1\left(\|\mathbf{e}^u\|_{L^{\infty}(J;\mathbf{V}_1(h))}+\|\mathbf{u}
      -\Pi_h\mathbf{u}\|_{L^{\infty}(J;\mathbf{V}_1(h))}\right)\right.\\
      &\qquad\left.+\mathcal{R}_2\left(\|e^{\varphi}\|_{L^{\infty}(J;V_2(h))}+\|\varphi-\Pi_h\varphi\|_{L^{\infty}(J;V_2(h))}\right)\right)\\
      \leq&\,\frac{1}{4}\left(\|\mathbf{e}^u\|^2_{L^{\infty}(J;\mathbf{V}_1(h))}+\|e^{\varphi}\|^2_{L^{\infty}(J;V_2(h))}\right)
      +Ch^{\min\{m,k\}}\left(\mathcal{R}_1^2+\|\mathbf{u}\|^2_{L^{\infty}(J;\mathbf{H}^{1+m} (\Omega))}\right)\\
      &\qquad+Ch^{\min\{m,k\}}\left(\mathcal{R}_2^2+\|\varphi\|^2_{L^{\infty}(J;H^{1+m} (\Omega_R))}\right),
    \end{align*}
    with the constant $C$ depends only on $C_{R}$, $C_{\text{coer}}$ and the constants $C_A$. Combining the above estimates for $T_1,T_2,T_3,T_4$ and $T_5$ then shows that
   \begin{align*}
     &\,\frac{1}{2}\|\mathbf{e}_t^u\|_{L^{\infty}(J;\mathbf{L}^2(\Omega))}^2 + \frac{1}{2}\|e_t^{\varphi}\|_{L^{\infty}(J;L^2(\Omega_R))}^2+\frac{1}{2}C_{\text{coer}}\|\mathbf{e}^u\|^2_{L^{\infty}(J;\mathbf{V}_1(h))}
     + \frac{1}{2}C_{\text{coer}}\|e^{\varphi}\|^2_{L^{\infty}(J;V_2(h))} \\
     \leq&\,\|\mathbf{e}_t^u(\mathbf{0})\|_{0,\Omega}^2+\|e_t^{\varphi}(0)\|_{0,\Omega_R}^2
     +C_{\text{cont}}\|\mathbf{e}^u(\mathbf{0})\|^2_h+C_{\text{cont}}\|e^{\varphi}(0)\|^2_h\\
     &\quad + Ch^{2\min\{m,k\}}\left(\|\mathbf{u}\|^2_{L^{\infty}(J;\mathbf{H}^{1+m}(\Omega))}
     + T^2\|\mathbf{u}_t\|^2_{L^{\infty}(J;\mathbf{H}^{1+m}(\Omega))}+\|\mathbf{u}_{tt}\|_{L^{1}(J;\mathbf{H}^m(\Omega))}\right)\\
     &\qquad + Ch^{2\min\{m,k\}}\left(\|\varphi\|_{L^{\infty}(J;H^{1+m}(\Omega_R))}
   +T^2\|\varphi_t\|^2_{L^{\infty}(J; H^{1+m}(\Omega_R))}+\|\varphi_{tt}\|^2_{L^{1}(J;H^m(\Omega_R))}\right)
   \end{align*}
   with a constant that is independent of $T$ and the mesh size. This concludes the proof of Theorem~\ref{theorem3.1}.\qed

   Hence, with a standard projection approximation for the initial condition, Theorem~\ref{theorem3.1} yields an optimal convergence estimate in the (DG) energy norm
   \[
     \|\mathbf{e}^u_t\|_{L^{\infty}(J;L^2(\Omega))} + \|\mathbf{e}^u\|_{L^{\infty}(J;\mathbf{V}_1)} \leq Ch^{\min\{m,k\}},
   \]
   and
   \[
     \|e^\varphi_t\|_{L^{\infty}(J;L^2(\Omega_R))}+ \|e^\varphi\|_{L^{\infty}(J;V_2(h))} \leq C h^{\min\{m,k\}},
   \]
   with a constant $C$ that is independent of $h$.

\section{Numerical results}
 We will present  the numerical results to verify the theoretical error analysis, and the IPDG penalty parameters are chosen as follows: $\alpha=100, \beta=1$. The discretization of the FSI problem in space by the IPDG method leads to the linear second-order system of ordinary differential equation as below
\begin{equation}\label{eq:timedis}
  M\ddot{U}^h(t)+N\dot{U}^h(t)+AU^h(t)=\mathbf{f}^h(t),\qquad\qquad t\in[0,T]
\end{equation}
with the initial conditions
\begin{equation}\label{eq:timeinitial}
  MU^h(0)=U_0^h,\quad M\dot{U}^h(0)=\mathbf{v}_0^h,
\end{equation}
Here $M$, $N$ and $A$ are the known coefficient  matrices. We use the second order Newmark time stepping scheme~\cite{NMN1959} to discretize (\ref{eq:timedis}) in time domain, and $l$ denotes the time step size with $t_n=nl$. Then the Newmark method consists of finding $\{U^h_n\}_n$ to $U^h(t_n)$ such that
\begin{eqnarray}
  M\ddot{U}^h(0) &=& \mathbf{f}^h(0)-N\dot{U}^h(0)-AU^h(0)\\
  \nonumber
  (M+\gamma lN+\delta l^2A)\ddot{U}^h(n) &=& \mathbf{f}^h(n)-N[\dot{U}^h(n-1)+(1-\gamma)\delta l^2\ddot{U}^h(n-1)]-\\
  &&A[U^h(n-1)+l\dot{U}^h(n-1)+\frac{1-2\delta}{2}l^2\ddot{U}^h(n-1)] \\
  \dot{U}^h(n) &=& \dot{U}^h(n-1)+(1-\gamma)l\ddot{U}^h(n-1)+\gamma l\ddot{U}^h(n) \\
  \nonumber
  U^h(n) &=& U^h(n-1)+ l\dot{U}^h(n-1)+\frac{1-2\delta}{2}l^2\ddot{U}^h(n-1)+\\
  &&\delta l^2\ddot{U}^h(n)
\end{eqnarray}
for $n=1,\cdots,N-1$. Here, $\mathbf{f}_n:=\mathbf{f}(t_n)$, while $\delta\geq 0$ and $\gamma\geq1/2$ are free parameters that still can be chosen.
In our examples, we use the explicit second-order Newmark scheme, setting  $\gamma=1/2$ and $\delta=0$. In our accuracy tests, we consider the following two models with the parameters $c = 1$, $\rho_1 = 1$, $\rho_2=1$, $\lambda = 1$, $\mu = 1$. The computational domain $\Omega$ and $\Omega_R$ are approximated by uniform triangle elements and piecewise linear basis functions ($k=1$) are employed.
\paragraph{Example 1} Let $\Omega$ be a disk with radius $R_0 = 1$, a plane incident wave $\varphi^i=\cos(\bx\cdot\bd)\cos(t)$ is given, where the direction is $\bd=(1,0)$. Choosing the artificial boundary $\Gamma_R$ to be a circle with the radius  $R = 2$, sharing the same center of $\Gamma$,.

\par{}Let the final time $T=1$ and use the numerical solution $U^r$ with mesh size $h=0.0134$ as a reference solution. We set time step $l=h/20$ to ensure the stability of time discretization. The numerical errors $U^r-U^h$ and convergence order are presented in Table 1 for various mesh sizes. The numerical results confirm the expected rates of $k-$th order for the energy norm and $(k+1)-$th order for the $L^2$ norm.

\begin{table}[h]
\centering
\label{tab:E1}
\caption{Numerical errors when $T=1$ for Example 1.}
\begin{tabular}{p{1.5cm}p{4cm}p{2.5cm}p{3cm}p{2.5cm}}
\hline
h &Error in energy norm &Order &$L^2$ error &Order\\ \hline
0.415 &1.3071e+00 &- &6.0158e-01 &-\\
0.207 &6.9562e-01 &0.91 &2.1094e-01 &1.51\\
0.104 &3.0057e-01 &1.21 &5.8146e-02 &1.86\\
0.052 &1.3364e-01 &1.17 &1.4833e-02 &1.97  \\
0.026 &6.3287e-02 &1.08 &3.7408e-03 &1.99  \\ \hline
\end{tabular}
\end{table}

\paragraph{Example 2} We consider the FSI problem on a $L$-shaped domain $\Omega=(-1,1)^2\setminus [0,1)^2$. The incident wave is a point source at $(2,0)$:
      \[\varphi^i=\begin{cases}
      \sin(2\pi t)&\text{$0\leq t\leq0.5$},\\
       0&
      \text{$t\geq 0.5$}.
      \end{cases}\]

The artificial boundary $\Gamma_R$ is a circle centering at $(0,0)$ with radius $R=3$. Here, although solution $(\mathbf{u},\varphi)$ is continuous respect to time, it has a singular point $(0,0)$, that is $\mathbf{u}\in C^{\infty}(J;\mathbf{H}^{2/3}(\Omega))$, $\varphi\in C^{\infty}(J;H^{2/3}(\Omega_R))$.
\par{}Similarly, we let the numerical solution $U^r$ on the mesh with mesh size $h=0.0227$ be the reference solution and $l=h/20$, and compute the  errors $U^r-U^h$ in the energy norm and $L^2$ norm at $T=1$. In this case, the parameter $m$ in Theorem~\ref{theorem3.1} is $2/3$, so the theoretical convergence rate in energy norm should be $2/3$. The results in Table 2 validate this conclusion.

\begin{table}[tbp]
\centering
\label{tab:E2}
\caption{Numerical errors when $T=1$ for Example 2.}
\begin{tabular}{p{1.5cm}p{4cm}p{2.5cm}p{3cm}p{2.5cm}}
\hline
h &Error in energy norm  &Order &$L_2$ error &Order\\ \hline
0.726 &3.9474e-01 &- &4.0360e-02 &-\\
0.207 &2.5512e-01 &0.63 &1.4849e-02 &1.45\\
0.104 &1.6564e-01 &0.65 &5.5443e-03 &1.41\\
0.052 &1.0484e-01 &0.66 &2.1281e-03 &1.38  \\
0.026 &6.6348e-02 &0.66 &8.2907e-04 &1.36  \\ \hline
\end{tabular}
\end{table}

\section{Conclusion}
\paragraph{}
In this work, we have proposed an interior penalty discontinuous Galerkin method for solving the acoustic-elastic wave interaction problem. A low-order approximate  absorbing boundary condition has been used to deal with the acoustic waves in the unbounded domain. Essential analysis, including a priori error analysis, have been performed for the discontinuous Galerkin solution. We will consider the analysis of stable and conserved fully discrete schemes for the interaction problem, and the interaction problem of wave propagation in orthotropic porous elastic media in the future. Numerical schemes using more accurate coupling method of discontinuous Galerkin and boundary integral equation methods will also be envisioned in our future work.
\bigskip

\section*{Appendix.  Proof of Theorem~\ref{thm31}}
\counterwithout{equation}{section}
\counterwithout{theorem}{section}
\par{}This appendix is devoted to the proof of Theorem~\ref{thm31}.  As we mentioned in Section 3, such a proof is technically complicated.  Here, we prove this theorem by using \textit{a priori} estimate of an elliptic problem and the abstract inversion theorem of the Laplace transform. The following lemma (Treve~\cite{TF1975}, Theorem 43.1) is the analog of the Paley-Wiener-Schiwarz theorem for the Fourier transform of the distributions with compact support in the case of the Laplace transform.
      \begin{lemma}
      \label{lem32}
       {Let $\mathbf{h}(s)$ denote a holomorphic function in the half-plane $Re(s) > \alpha_0$, valued in the Banach space $E$. The following conditions are equivalent:\\
      (i) there is a distribution $T\in D'_+(E)$ whose Laplace transform is equal to $\mathbf{h}(s)$;\\
      (ii) there is a real $\alpha_1$, $\alpha_0\leq\alpha_1<\infty$, a constant $C > 0$, and an integer $m \geq 0$ such that, for all complex numbers $s$ with $Re(s) > \alpha_1$, it holds that
      \[
        \|\mathbf{h}(s)\|_E\leq C(1+s)^m,
      \]
      where $\|\cdot\|_E$ is the norm of Banach space $E$, and $D'_+$ is the space of distributions on the real line which vanish identically in the open negative half-line.}
      \end{lemma}
      \paragraph{}Take the Laplace transform of \eqref{eq:rdfsi1}-\eqref{eq:rdfsi5} and denote $\tilde{\mathbf{u}}=\mathcal{L}(\mathbf{u})$, $\tilde{\varphi}=\mathcal{L}(\varphi)$, where $\mathcal{L}$ is the Laplace transform operator, we obtain:
      \begin{align}
      \rho_2 s^2\tilde{\mathbf{u}}-\text{div}(\sigma(\tilde{\mathbf{u}}))&= \mathbf{0}, && \text{~in~} \Omega,\label{eq:rdfsiL1}\\
      \frac{s^2}{c^2}\tilde{\varphi}-\Delta\tilde{\varphi} &= 0, && \text{~in~}\Omega_{R},\label{eq:rdfsiL2}\\
      \sigma(\tilde{\mathbf{u}})\mathbf{n}&=-(\rho_1s \tilde{\varphi}+\rho_1s\tilde{\varphi})\mathbf{n}, && \text{~on~}\Gamma,\label{eq:rdfsiL3}\\
      \frac{\partial \tilde{\varphi}}{\partial\mathbf{n}} &= -\frac{\partial \tilde{\varphi}^i}{\partial\mathbf{n}}-s\tilde{\mathbf{u}}\cdot\mathbf{n},&&\text{~on~} \Gamma,\label{eq:rdfsiL4}\\
      \frac{\partial \tilde{\varphi}}{\partial \mathbf{n}}&=-s\tilde{\varphi},&&\text{~on~} \Gamma_R.\label{eq:rdfsiL5}
      \end{align}

    \begin{lemma}
      \label{lem33}
      {Let $s=s_1+is_2$, $s_1>0$, $s_2\in \mathbb{R}$. For any $g\in L^2(\Omega_R)$, let $(\tilde{\mathbf{u}},\tilde{\varphi})$ be the solution of \eqref{eq:rdfsiL1}-\eqref{eq:rdfsiL5}. Then there exists a constant $C$ independent of $s$ such that:
      \begin{eqnarray}\label{eq:lemma1}
     \|\varepsilon(\tilde{\mathbf{u}})\|^2_{F(\Omega)}+\|\nabla\cdot \tilde{\mathbf{u}}\|^2_{0,\Omega}+\|s\tilde{\mathbf{u}}\|^2_{0,\Omega} \leq C\frac{(1+|s|)^2}{s^2_1}\left\|\tilde{\varphi}^i\right\|^2
      _{\frac{1}{2},\Gamma},\label{eq:lemma1.1}\\
      \|\nabla \tilde{\varphi}\|^2_{0,\Omega_R}+\|s\tilde{\varphi}\|^2_{0,\Omega_R} \leq C\frac{(1+|s|)^2}{s_1^2}\left\|\tilde{\varphi}^i\right\|^2
      _{\frac{1}{2},\Gamma}.\label{eq:lemma1.2}
      \end{eqnarray}}
      \end{lemma}
     \begin{proof}
       Multiplying~\eqref{eq:rdfsiL1} and~\eqref{eq:rdfsiL2} by test functions $\mathbf{v}\in \mathbf{H}^1(\Omega)$ and  $\phi\in H^1(\Omega_R)$ respectively, then using  integration by parts, the symmetry of the stress tensor, and the transmission and boundary conditions, we get
      \begin{align}
      \nonumber
      & \int_{\Omega}\left(\rho_2s^2\tilde{\mathbf{u}}\cdot\mathbf{v}+\lambda(\nabla\cdot\tilde{\mathbf{u}})(\nabla\cdot \mathbf{v})+2\mu(\varepsilon(\tilde{\mathbf{u}}):\varepsilon(\mathbf{v}))\right)\,d\mathbf{x} \\
      &\qquad\qquad\qquad\qquad +\int_{\Gamma}\rho_1s(\tilde{\varphi}\mathbf{n})\cdot\mathbf{v}\,dS =-\int_{\Gamma}\rho_1s\left(\tilde{\varphi}^i\mathbf{n}\right)\cdot\mathbf{v}\,dS,\label{eq:THVF0.1}\\
      \nonumber
      & \int_{\Omega_R}\left(\frac{s}{c^2}\tilde{\varphi}\cdot\phi+\frac{1}{s}\nabla\tilde{\varphi}\cdot\nabla\phi\right)\,d\mathbf{x} - \int_{\Gamma}(\tilde{\mathbf{u}}\cdot\mathbf{n})\phi\,dS \\
      &\qquad\qquad\qquad\qquad + \int_{\Gamma_R}\tilde{\varphi}\phi\,dS  = \int_{\Gamma}\frac{1}{s}\left(\frac{\partial\tilde{\varphi}^i}{\partial\mathbf{n}}\right)\cdot\phi\,dS\label{eq:THVF0.2}.
      \end{align}
      Multiply~\eqref{eq:THVF0.1} with $\frac{1}{\rho_1s}$ and add it to \eqref{eq:THVF0.2}, we get:
      \begin{eqnarray}\label{eq:THVF1}
      \tilde{a}(\tilde{\mathbf{u}},\tilde{\varphi};\mathbf{v},\phi)=\tilde{L}(\mathbf{v},\phi)
      \end{eqnarray}
      where
      \begin{align}
      \nonumber
      \tilde{a}(\tilde{\mathbf{u}},\tilde{\varphi};\mathbf{v},\phi)&:=\int_{\Omega}\left( \frac{\rho_2s}{\rho_1}\tilde{\mathbf{u}}\cdot\mathbf{v}+\frac{\lambda}{\rho_1s}(\nabla\cdot\tilde{\mathbf{u}})(\nabla\cdot \mathbf{v})+\frac{2\mu}{\rho_1s}(\varepsilon(\tilde{\mathbf{u}}):\varepsilon(\mathbf{v}))\right)\,d\mathbf{x}\\
      \nonumber
      &\quad + \int_{\Omega_R} \left(\frac{s}{c^2}\tilde{\varphi}\cdot\phi+\frac{1}{s}\nabla\tilde{\varphi}\cdot\nabla\phi\right)\,d\mathbf{x} +\int_{\Gamma}(\tilde{\varphi}\mathbf{n})\cdot\mathbf{v}\,dS\\
      &\quad - \int_{\Gamma}(\tilde{\mathbf{u}}\cdot\mathbf{n})\phi\,dS+\int_{\Gamma_R}\tilde{\varphi}\phi \,dS,\label{eq:THVF2.1}\\
      \tilde{L}(\mathbf{v},\phi)&:= \int_{\Gamma}\frac{1}{s}\left(\frac{\partial\tilde{\varphi}^i}{\partial\mathbf{n}}\right)\cdot\phi\,dS-\int_{\Gamma}(\tilde{\varphi}^i\mathbf{n})\cdot\mathbf{v}\,dS.\label{eq:THVF2.2}
      \end{align}
      Using Cauchy-Schwarz inequality and the trace theorem, we have:
      \begin{align*}
      \left|\tilde{a}(\tilde{\mathbf{u}},\tilde{\varphi};\mathbf{v},\phi)\right| &\leq \left|\frac{\rho_2s}{\rho_1}\right|\|\tilde{\mathbf{u}}\|_{0,\Omega}\|\mathbf{v}\|_{0,\Omega} +\left|\frac{\lambda}{\rho_1s}\right|\|\nabla\cdot\tilde{\mathbf{u}}\|_{F(\Omega)}\|\nabla\cdot \mathbf{v}\|_{F(\Omega)}\\
      &\quad + \left|\frac{2\mu}{\rho_1s}\right|\|\varepsilon(\tilde{\mathbf{u}})\|_{F(\Omega)}\|\varepsilon(\mathbf{v})\|_{F(\Omega)} +\left|\frac{s}{c^2}\right|\|\tilde{\varphi}\|_{0,\Omega_R}\|\phi\|_{0,\Omega_R}\\
      &\quad + \left|\frac{1}{s}\right|\|\nabla\tilde{\varphi}\|_{0,\Omega_R}\|\nabla\phi\|_{0,\Omega_R} +\|\tilde{\varphi}\mathbf{n}\|_{-\frac{1}{2},\Gamma}\|\mathbf{v}\|_{\frac{1}{2},\Gamma}\\
      &\quad +\|\tilde{\mathbf{u}}\cdot\mathbf{n}\|_{-\frac{1}{2},\Gamma}\|\phi\|_{\frac{1}{2},\Gamma} +\|\tilde{\varphi}\|_{-\frac{1}{2},\Gamma_R}\|\phi\|_{\frac{1}{2},\Gamma_R}\\
      &\leq C_1\left(\|\tilde{\mathbf{u}}\|_{1,\Omega}\|\mathbf{v}\|_{1,\Omega}+\|\tilde{\varphi}\|_{1,\Omega_R}\|\phi\|_{1,\Omega_R} +\|\tilde{\varphi}\|_{\frac{1}{2},\Gamma}\|\mathbf{v}\|_{\frac{1}{2},\Gamma}\right.\\
      &\qquad +\left.\|\tilde{\mathbf{u}}\|_{\frac{1}{2},\Gamma}\|\phi\|_{\frac{1}{2},\Gamma}+\|\tilde{\varphi}\|_{\frac{1}{2},\Gamma_R} \|\phi\|_{\frac{1}{2},\Gamma_R}\right)\\
      &\leq C_1\left(\|\tilde{\mathbf{u}}\|_{1,\Omega}\|\mathbf{v}\|_{1,\Omega}+\|\tilde{\varphi}\|_{1,\Omega_R}\|\phi\|_{1,\Omega_R} +\|\tilde{\varphi}\|_{1,\Omega_R}\|\mathbf{v}\|_{1,\Omega}\right.\\
      &\qquad +\left.\|\tilde{\mathbf{u}}\|_{1,\Omega}\|\phi\|_{1,\Omega_R}+\|\tilde{\varphi}\|_{1,\Omega_R}\|\phi\|_{1,\Omega_R}\right),
      \end{align*}
      where $C_1$ is a constant that independent of $\tilde{\mathbf{u}},\tilde{\varphi},\mathbf{v},\phi$, hence the sesquilinear form $\tilde{a}(\tilde{\mathbf{u}},\tilde{\varphi};\mathbf{v},\phi)$ is bounded.
      \par{}Letting $(\mathbf{v},\phi)=(\tilde{\mathbf{u}},\tilde{\varphi})$ in \eqref{eq:THVF2.1} yields:
      \begin{align}
      \nonumber
      \tilde{a}(\tilde{\mathbf{u}},\tilde{\varphi};\tilde{\mathbf{u}},\tilde{\varphi})& = \int_{\Omega}\left( \frac{\rho_2s}{\rho_1}|\tilde{\mathbf{u}}|^2+\frac{\lambda}{\rho_1s}|\nabla\cdot\tilde{\mathbf{u}}|^2 +\frac{2\mu}{\rho_1s}(\varepsilon(\tilde{\mathbf{u}}):\varepsilon(\tilde{\mathbf{u}}))\right)\,d\mathbf{x}\\
      &\qquad + \int_{\Omega_R} \left(\frac{s}{c^2}|\tilde{\varphi}|^2+\frac{1}{s}|\nabla\tilde{\varphi}|^2\right)\,d\mathbf{x}
      +\int_{\Gamma_R}|\tilde{\varphi}|^2\,dS.\label{eq:THVF4.1}
      \end{align}
      Taking real part of \eqref{eq:THVF4.1}, we have
      \begin{align}
       \nonumber
       &\,\text{Re}\left(\tilde{a}(\tilde{\mathbf{u}},\tilde{\varphi};\tilde{\mathbf{u}},\tilde{\varphi})\right) \\
       \nonumber
      = &\quad \frac{\rho_2s_1}{\rho_1}\|\tilde{\mathbf{u}}\|_{0,\Omega}^2+\frac{\lambda s_1}{\rho_1|s|^2}\|\nabla\cdot\tilde{\mathbf{u}}\|_{0,\Omega}^2 +\frac{2\mu s_1}{\rho_1|s|^2}\|\varepsilon(\tilde{\mathbf{u}})\|_{F(\Omega)}^2 \\ &\qquad\qquad +\frac{s_1}{c^2}\|\tilde{\varphi}\|_{0,\Omega_R}^2+\frac{s_1}{|s|^2}\|\nabla\tilde{\varphi}\|_{0,\Omega_R}^2
      +\int_{\Gamma_R}|\tilde{\varphi}|^2 \,ds\label{eq:THVF5}\\
       \nonumber
      \geq &\quad C_2\frac{s_1}{|s|^2}\left(|s|^2\|\tilde{\mathbf{u}}\|_{0,\Omega}^2+\|\nabla\cdot\tilde{\mathbf{u}}\|_{0,\Omega}^2 +\|\varepsilon(\tilde{\mathbf{u}})\|_{F(\Omega)}^2+|s|^2\|\tilde{\varphi}\|_{0,\Omega_R}^2
      +\|\nabla\tilde{\varphi}\|_{0,\Omega_R}^2\right).
      \end{align}
       where $C_2$ is a constant that independent of $s$.  According to Lax-Milgram lemma, there exists a unique solution $(\tilde{\mathbf{u}},\tilde{\varphi})\in \mathbf{H}^1(\Omega)\times H^1(\Omega_R),$ satisfies the following problem
      \begin{equation}\label{eq:THVF6}
        \tilde{a}(\tilde{\mathbf{u}},\tilde{\varphi};\mathbf{v},\phi)=\tilde{L}(\mathbf{v},\phi),\qquad \forall\,(\mathbf{v},\phi)\in \mathbf{H}^1(\Omega)\times H^1(\Omega_R)
      \end{equation}
      \par{}Based on equation~\eqref{eq:THVF6} and Cauchy-Schwarz inequality, there exists a constant $C_3$ such that
      \begin{align}
       \nonumber
      |\tilde{a}(\tilde{\mathbf{u}},\tilde{\varphi};\tilde{\mathbf{u}},\tilde{\varphi})|=|L(\tilde{\mathbf{u}},\tilde{\varphi})|&\leq \frac{1}{|s|^2}\left\|\nabla_{\mathbf{n}}\tilde{\varphi}^i\right\|_{-\frac{1}{2},\Gamma}\|s\tilde{\varphi}\|_{\frac{1}{2},\Gamma} +\frac{1}{|s|^2}\left\|s\tilde{\varphi}^i\cdot\mathbf{n}\right\|_{-\frac{1}{2},\Gamma}\|s\tilde{\mathbf{u}}\|_{\frac{1}{2},\Gamma}\\
      \nonumber
      &\leq C_3\left(\frac{1}{|s|^2}\left\|\tilde{\varphi}^i\|_{\frac{1}{2},\Gamma}\right\|s\tilde{\varphi}\|_{1,\Omega_R} +\frac{1}{|s|^2}\left\|\tilde{\varphi}^i\right\|_{\frac{1}{2},\Gamma}\left\|s^2\tilde{\mathbf{u}}\right\|_{1,\Omega}\right)\\
      \nonumber
      &\leq C_3\frac{1+|s|}{|s|^2}\left\|\tilde{\varphi}^i\right\|_{\frac{1}{2},\Gamma}\left(\|s\tilde{\varphi}\|^2_{0,\Omega_R}+ \|\nabla\tilde{\varphi}\|^2_{0,\Omega_R}\right.\\
      &\qquad\left.+\|s\tilde{\mathbf{u}}\|^2_{1,\Omega}+\|\nabla\cdot\tilde{\mathbf{u}}\|_{0,\Omega}^2 +\|\varepsilon(\tilde{\mathbf{u}})\|_{F(\Omega)}^2\right)^{1/2}\label{eq:THVF7}.
      \end{align}
      Combining~\eqref{eq:THVF5} and~\eqref{eq:THVF7}:
      \begin{eqnarray}\label{THVF8}
      \nonumber
      &\frac{s_1}{|s|^2}C_2\left(|s|^2\|\tilde{\mathbf{u}}\|_{0,\Omega}^2+\|\nabla\cdot\tilde{\mathbf{u}}\|_{0,\Omega}^2 +\|\varepsilon(\tilde{\mathbf{u}})\|_{F(\Omega)}^2+|s|^2\|\tilde{\varphi}\|_{0,\Omega_R}^2+\|\nabla\tilde{\varphi}\|_{0,\Omega_R}^2\right)\\
      &\leq C_3\frac{1+|s|}{|s|^2}\|\tilde{\varphi}^i\|_{\frac{1}{2},\Gamma}\left(\|s\tilde{\varphi}\|^2_{0,\Omega_R}+ \|\nabla\tilde{\varphi}\|^2_{0,\Omega_R} +\|s\tilde{\mathbf{u}}\|^2_{0,\Omega}+\|\nabla\cdot\tilde{\mathbf{u}}\|_{0,\Omega}^2 +\|\varepsilon(\tilde{\mathbf{u}})\|_{F(\Omega)}^2\right)^{1/2}.
      \end{eqnarray}
      Letting $C=\left(\frac{C_3}{C_2}\right)^2$, we obtain \eqref{eq:lemma1.1} and \eqref{eq:lemma1.2}.        \end{proof}


Now we are ready to prove Theorem~\ref{thm31}.
      \paragraph{Proof of Theorem~\ref{thm31}:} Using Lemma~\ref{lem33}, $\tilde{\mathbf{u}},\tilde{\varphi}$ are the holomorphic functions of $s$ on the half plane $s_1>\gamma>0$. Hence the inverse Laplace transform of $\tilde{\mathbf{u}}$ and $\tilde{\varphi}$ exist i.e., the problem~\eqref{eq:rdfsi1}-\eqref{eq:rdfsi6} has a unique solution $(\mathbf{u},\varphi)=(\mathcal{L}^{-1}(\tilde{\mathbf{u}}),\mathcal{L}^{-1}(\tilde{\varphi}))$.
      \par{} For $\varphi$,we have:
      \begin{eqnarray}
      \nonumber
      \int_0^{T}\left(\|\nabla\varphi\|_{0,\Omega_R}^2+\|\varphi_t\|_{0,\Omega_R}^2\right)dt &\leq& \int_0^{T}e^{-2s_1(t-T)}\left(\|\nabla\varphi\|_{0,\Omega_R}^2+\|\varphi_t\|_{0,\Omega_R}^2\right)\,dt \\
      \nonumber
      &=& e^{2s_1 T}\int_0^{T}e^{-2s_1t}\left(\|\nabla\varphi\|_{0,\Omega_R}^2+\|\varphi_t\|_{0,\Omega_R}^2\right)\,dt \\
      &\leq& C \int_0^{\infty}e^{-2s_1t}\left(\|\nabla\varphi\|_{0,\Omega_R}^2+\|\varphi_t\|_{0,\Omega_R}^2\right)\,dt
      \end{eqnarray}
      Similarly, for $\mathbf{u}$, we have:
      \begin{eqnarray}
      \nonumber
    &\quad& \int_0^{T}\left(\|\nabla\cdot\mathbf{u}\|_{0,\Omega}^2+\|\varepsilon(\mathbf{u})\|_{F(\Omega)}^2+\|\mathbf{u}_t\|_{0,\Omega)}^2\right)\,dt \\
    &\leq & C \int_0^{\infty}e^{-2s_1t}\left(\left\|\nabla\cdot\mathbf{u}\right\|_{0,\Omega}^2+\|\varepsilon(\mathbf{u})\|_{F(\Omega)}^2+\|\mathbf{u}_t\|_{0,\Omega}^2\right)\,dt.
      \end{eqnarray}

      \par{}From estimate~\eqref{eq:lemma1.2} and the trace theorem, we have:
      \begin{eqnarray}\label{eq:theorem1}
      \|\nabla\cdot \tilde{\mathbf{u}}\|^2_{0,\Omega}+\|\varepsilon(\tilde{\mathbf{u}})\|^2_{F(\Omega)}+\|s\tilde{\mathbf{u}}\|^2_{0,\Omega} \leq C\frac{(1+|s|)^2}{s_1^2}\left\|\tilde{\varphi}^i\right\|^2
      _{\frac{1}{2},\Gamma}\leq C\frac{(1+|s|)^2}{s_1^2}\left\|\tilde{\varphi}^i\right\|^2
      _{1,\Omega_R},\label{eq:theorem1.1}\\
      \|\nabla \tilde{\varphi}\|^2_{0,\Omega}+\|s\tilde{\varphi}\|^2_{0,\Omega} \leq C\frac{(1+|s|)^2}{s_1^2}\left\|\tilde{\varphi}^i\right\|^2
      _{\frac{1}{2},\Gamma}\leq C\frac{(1+|s|)^2}{s_1^2}\left\|\tilde{\varphi}^i\right\|^2
      _{1,\Omega_R}.\label{eq:theorem1.2}
      \end{eqnarray}
      Using Parseval identity (see~\cite{Cohen:2007}) and estimate~\eqref{eq:theorem1.1}:
      \begin{eqnarray}\label{eq:theorem2}
      \nonumber
      &~&\int_0^{\infty}e^{-2s_1t}\left(\|\nabla\cdot\mathbf{u}\|_{0,\Omega}^2+\|\varepsilon(\mathbf{u})\|_{F(\Omega)}^2 +\|\mathbf{u}_t\|_{0,\Omega}^2\right)\,dt\\
      \nonumber
      &=&\frac{1}{2\pi}\int_{-\infty}^{\infty}\left(\|\nabla\cdot\tilde{\mathbf{u}}\|_{0,\Omega}^2+\|\varepsilon(\tilde{\mathbf{u}})\|_{F(\Omega)}^2 +\|s\tilde{\mathbf{u}}\|^2_{0,\Omega}\right)\,ds_2\\
      \nonumber
      &\leq& C s_1^{-2}\int_{-\infty}^{\infty}\left(\left\|\tilde{\varphi}^i\right\|^2_{1,\Omega_R}+\left\|s\tilde{\varphi}^i\right\|^2_{1,\Omega_R}\right)\,ds_2\\
      \nonumber
      &=& C s_1^{-2}\int_{-\infty}^{\infty}\left(\left\|\mathcal{L}(\varphi^i)\right\|^2_{1,\Omega_R}+\left\|\mathcal{L}(\varphi_t^i)\right\|^2_{1,\Omega_R}\right)\,ds_2\\
      &\leq& C s_1^{-2}\int_{0}^{\infty}e^{-2s_1t}\left(\left\|\varphi^i\right\|^2_{1,\Omega_R}+\left\|\varphi_t^i\right\|^2_{1,\Omega_R}\right)\,dt,
      \end{eqnarray}
      which shows that
      \[
        \mathbf{u}(\mathbf{x},t)\in L^2\left(J;\mathbf{H^1}(\Omega)\right)\cap H^1\left(J;\mathbf{L^2}(\Omega)\right).
      \]
      Similarly, we have
      \begin{eqnarray}\label{eq:theorem3}
        \int_0^{\infty}e^{-2s_1t}\left(\|\nabla\varphi\|_{0,\Omega_R}^2+\|\varphi_t\|_{0,\Omega_R}^2\right)\,dt
        \leq C s_1^{-2}\int_{0}^{\infty}e^{-2s_1t}\left(\left\|\varphi^i\right\|^2_{1,\Omega_R}+\left\|\varphi_t^i\right\|^2_{1,\Omega_R}\right)\,dt.
      \end{eqnarray}
      It follows from \eqref{eq:theorem3} that
      \[
        \varphi(x,t)\in L^2\left(J;H^1(\Omega_R)\right)\cap H^1\left(J;L^2(\Omega_R)\right).
      \]
      \par{}Next we prove the stability, which also helps establishing numerical stability of the IPDG scheme in the following section.  For any $0<t<T$, consider the energy function:
      \begin{equation}
        \label{eq:stable1}
            E(t) = E_1(t) + E_2(t)
      \end{equation}
      where
      \[
        E_1(t) = \left\|\frac{\sqrt{\rho_2}}{\sqrt{\rho_1}}\mathbf{u}_t\right\|^2_{0,\Omega}+\left\|\frac{\sqrt{\lambda}}{\sqrt{\rho_1}}\nabla\cdot\mathbf{u}\right\|^2_{0,\Omega} +\left\|\frac{\sqrt{2\mu}}{\sqrt{\rho_1}}\varepsilon(\mathbf{u})\right\|^2_{F(\Omega)}
      \]
      and
      \[
         E_2(t) = \left\|\frac{1}{c}\varphi_t\right\|^2_{0,\Omega_R}+\left\|\nabla\varphi\right\|^2_{0,\Omega_R}.
       \]
      Obviously,
      \begin{equation}\label{eq:stable2}
        E(t)-E(0)=\int_0^tE'(\tau)d\tau=\int_0^t\left(E_1'(\tau)+E_2'(\tau)\right)\,d\tau.
      \end{equation}
      By using the integration by parts and~\eqref{eq:rdfsi1},\eqref{eq:rdfsi3},\eqref{eq:rdfsi6} that
      \begin{eqnarray}\label{eq:stable3}
       \nonumber
       \int_0^tE_1'(\tau)\,d\tau &=& 2\int_0^t\int_\Omega\left(\frac{\rho_2}{\rho_1}\mathbf{u}_{tt}\cdot\mathbf{u}_t +\frac{\lambda}{\rho_1}(\nabla\cdot\mathbf{u}_t)(\nabla\cdot\mathbf{u})+\frac{2\mu}{\rho_1}\varepsilon(\mathbf{u}_t):\varepsilon(\mathbf{u})\right)\,d\mathbf{x}d\tau \\
        \nonumber
        &=&2\int_0^t\int_\Omega\left(\frac{1}{\rho_1}\nabla\cdot\sigma(\mathbf{u})\cdot\mathbf{u}_t +\frac{\lambda}{\rho_1}(\nabla\cdot\mathbf{u}_t)(\nabla\cdot\mathbf{u})+\frac{2\mu}{\rho_1}\varepsilon(\mathbf{u}_t):\varepsilon(\mathbf{u})\right)\,d\mathbf{x}d\tau \\
        \nonumber
        &=&2\int_0^t\int_\Omega\left(-\frac{1}{\rho_1}[\lambda(\nabla\cdot\mathbf{u})(\nabla\cdot\mathbf{u}_t)+2\mu\varepsilon(\mathbf{u}):\varepsilon(\mathbf{u}_t) ]+\frac{\lambda}{\rho_1}(\nabla\cdot\mathbf{u}_t)(\nabla\cdot\mathbf{u})\right.\\
        \nonumber
        &\,&\qquad\left.+\frac{2\mu}{\rho_1}\varepsilon(\mathbf{u}_t):\varepsilon(\mathbf{u})\right)\,d\mathbf{x}d\tau +2\int_0^t\int_{\Gamma}\frac{1}{\rho_1}\sigma(\mathbf{u})\cdot\mathbf{n}\cdot\mathbf{u}_t \,d\mathbf{x}d\tau\\
        &=&-2\int_0^t\left(\langle\varphi_t\mathbf{n},\mathbf{u}_t\rangle_{\Gamma}+\langle\varphi_t^i\mathbf{n},\mathbf{u}_t\rangle_{\Gamma}\right)\,d\tau.
      \end{eqnarray}
      Similarly, follows from the integration by parts and~\eqref{eq:rdfsi1},\eqref{eq:rdfsi3},\eqref{eq:rdfsi6}, we have
      \begin{eqnarray}\label{eq:stable4}
       \nonumber
       \int_0^tE_2'(\tau)\,d\tau &=& 2\int_0^t\int_\Omega\left(\frac{1}{c^2}\varphi_{tt}\varphi_t +\nabla(\varphi_t)\nabla\varphi\right)\,dxd\tau \\
        \nonumber
        &=&2\int_0^t\int_\Omega\left(\Delta\varphi\varphi_t +\nabla(\varphi_t)\nabla\varphi\right)\,dxd\tau \\
        \nonumber
        &=&2\int_0^t\int_\Omega\left(-\nabla\varphi\nabla(\varphi_t)+\nabla(\varphi_t)\nabla\varphi\right)\,dxd\tau \\
        \nonumber
        &\,&\qquad +2\int_0^t\int_{\Gamma_R}(\nabla\varphi\mathbf{n})\varphi_t \,dxd\tau
        -2\int_0^t\int_{\Gamma}(\nabla\varphi\mathbf{n})\varphi_t \,dxd\tau\\
        &=&2\int_0^t\left(\langle-\varphi_t,\varphi_t\rangle_{\Gamma_R} \,d\tau+2\int_0^t\langle\mathbf{u}_t\cdot\mathbf{n},\varphi_t\rangle_{\Gamma} +\langle\nabla\varphi^i\mathbf{n},\varphi_t\rangle_{\Gamma}\right)\,d\tau.
      \end{eqnarray}
      Since $E(0)=0$, by combining equations~\eqref{eq:stable2}-\eqref{eq:stable4} and using the trace theorem, we have:
      \begin{eqnarray}\label{eq:stable5}
       \nonumber
       E(t)&=&2\int_0^t\left(\langle-\varphi_t,\varphi_t\rangle_{\Gamma_R}+\langle\nabla\varphi^i\mathbf{n},\varphi_t\rangle_{\Gamma} -\langle\varphi_t^i\mathbf{n},\mathbf{u}_t\rangle_{\Gamma}\right)\,d\tau\\
        \nonumber
        &\leq&2\int_0^t\left(\langle\nabla\varphi^i\mathbf{n},\varphi_t\rangle_{\Gamma} -\langle\varphi^i_t\mathbf{n},\mathbf{u}_t\rangle_{\Gamma}\right)\,d\tau
        \leq 2\int_0^t \left(\|\nabla\varphi^i\|_{-\frac{1}{2},\Gamma}\|\varphi_t\|_{\frac{1}{2},\Gamma}+\|\varphi_t^i\|_{-\frac{1}{2},\Gamma}\|\mathbf{u}_t\|_{\frac{1}{2},\Gamma}\right)\,d\tau\\
        \nonumber
        &\leq&2\int_0^t \left(\|\nabla\varphi^i\|_{-\frac{1}{2},\Gamma}\|\varphi_t\|_{1,\Omega_R}+\|\varphi_t^i\|_{-\frac{1}{2},\Gamma}\|\mathbf{u}_t\|_{1,\Omega}\right)\,d\tau\\
        &\leq&C\left( \|\nabla\varphi^i\|_{L^1(J;H^{-1/2}(\Gamma))}\max_{t\in[0,T]}\|\varphi_t\|_{1,\Omega_R}+\|\varphi_t^i\|_{L^1(J;H^{-1/2}(\Gamma))}\max_{t\in[0,T]}\|\mathbf{u}_t\|_{1,\Omega} \right).
      \end{eqnarray}
      Therefore, using Young's inequality, we obtain:
      \begin{eqnarray}\label{eq:stable6}
       \nonumber
       &\,&\left(\|\mathbf{u}_t\|^2_{0,\Omega}+\|\nabla\cdot\mathbf{u}\|^2_{0,\Omega} +\|\varepsilon(\mathbf{u})\|^2_{F(\Omega)}\right)+\left(\|\nabla\varphi\|^2_{0,\Omega_R} +\|\varphi_t\|^2_{0,\Omega_R}\right) \leq C E(t) \\
        \nonumber
        &\leq& C\left(\epsilon\max_{t\in[0,T]}\left(\|\mathbf{u}_t\|^2_{0,\Omega}+\|\nabla\cdot\mathbf{u}_t\|^2_{0,\Omega} +\|\varepsilon(\mathbf{u}_t)\|^2_{F(\Omega)}\right)+\epsilon\max_{t\in[0,T]}\left(\|\nabla\varphi_t\|^2_{0,\Omega_R} +\|\varphi_t\|^2_{0,\Omega_R}\right)\right.\\
       &\quad&\qquad\left.+\frac{1}{\epsilon}\left\|\nabla\varphi^i\right\|^2_{L^1(J;H^{-1/2}(\Gamma))} +\frac{1}{\epsilon} \left\|\varphi_t^i\right\|^2_{L^1(J;H^{-1/2}(\Gamma))}\right).
      \end{eqnarray}
      \par{}Since the norm on the right-hand side in~\eqref{eq:stable6} contains $\mathbf{u}_t$ and $\varphi_t$, which cannot be bounded by the left-hand side, here a new system is needed.  Take the first partial derivative of \eqref{eq:rdfsi1}-\eqref{eq:rdfsi6} with respect to $t$, we have:
      \begin{align}
        \rho_2\mathbf{u}_{ttt}-\text{div}(\sigma(\mathbf{u}_t))&= \mathbf{0}, &&\text{in~} \Omega\times J,\label{eq:stable7.1}\\
        \frac{1}{c^2}\varphi_{ttt}-\Delta \varphi_t &= 0, &&\text{in~}\Omega_R \times J,\label{eq:stable7.2}\\
        \sigma(\mathbf{u}_t)\mathbf{n}&=-\rho_1(\varphi_{tt}+\varphi_{tt}^i)\mathbf{n}, &&\text{on~} \Gamma\times J,\label{eq:stable7.3}\\
        \frac{\partial\varphi_t}{\partial\mathbf{n}} &= -\frac{\partial\varphi^i_t}{\partial\mathbf{n}}-\mathbf{u}_{tt}\cdot\mathbf{n},&& \text{on~} \Gamma\times J,\label{eq:stable7.4}\\
      \frac{\partial \varphi_t}{\partial\mathbf{n}}&= -\varphi_{tt},&&\text{on~} \Gamma_R\times J,\label{eq:stable7.5}\\
      \mathbf{u}_t|_{t=0}=\left.\mathbf{u}_{tt}\right|_{t=0} & =\mathbf{0},\, x \in \Omega,\qquad\quad\text{and} && \varphi_t|_{t=0}=\left.\varphi_{tt}\right|_{t=0}=0,\quad x \in \Omega_R. \label{eq:stable7.6}
      \end{align}
      We consider the energy function:
      \[
            F(t) = F_1(t) + F_2(t)
      \]
      where
      \[
        F_1(t) = \left\|\frac{\sqrt{\rho_2}}{\sqrt{\rho_1}}\mathbf{u}_{tt}\right\|^2_{0,\Omega}+\left\|\frac{\sqrt{\lambda}}{\sqrt{\rho_1}}\nabla\cdot\mathbf{u}_t\right\|^2_{0,\Omega} + \left\|\frac{\sqrt{2\mu}}{\sqrt{\rho_1}}\varepsilon(\mathbf{u}_t)\right\|^2_{F(\Omega)},
      \]
      and
      \[
        F_2(t) = \left\|\frac{1}{c}\varphi_{tt}\right\|^2_{0,\Omega_R}+\left\|\nabla\varphi_t\right\|^2_{0,\Omega_R}.
      \]
      It is clear that $F(0)=0$.  Similarly, follows the same step to prove inequality~\eqref{eq:stable5}, we obtain that
      \begin{eqnarray}\label{eq:stable9}
       \nonumber
       F(t)&=&2\int_0^t\left(\langle-\varphi_{tt},\varphi_{tt}\rangle_{\Gamma_R}+\langle\nabla\varphi_t^i\mathbf{n},\varphi_{tt}\rangle_{\Gamma} -\langle\varphi_{tt}^i\mathbf{n},\mathbf{u}_{tt}\rangle_{\Gamma}\right)\,d\tau\\
        \nonumber
        &\leq&2\int_0^t\left(\langle\nabla\varphi_{t}^i\mathbf{n},\varphi_{tt}\rangle_{\Gamma} -\langle\varphi^i_{tt}\mathbf{n},\mathbf{u}_{tt}\rangle_{\Gamma}\right)\,d\tau\\
        \nonumber
        &=&2\left.\int_{\Gamma}\left(\nabla\varphi_{t}^i\mathbf{n}\,\varphi_{t} -\varphi^i_{tt}\mathbf{n}\,\mathbf{u}_{t}\right)\right|_{0}^{t}-2\int_0^t\left(\langle\nabla\varphi_{tt}^i\mathbf{n},\varphi_{t}\rangle_{\Gamma} -\langle\varphi^i_{ttt}\mathbf{n},\mathbf{u}_{t}\rangle_{\Gamma}\right)\,d\tau\\
        \nonumber
        &\leq&C\left(\max_{t\in[0,T]}\|\mathbf{u}_{t}\|_{1,\Omega}\left(\max_{t\in[0,T]}\left\|\varphi_{tt}^i\right\|_{-\frac{1}{2},\Gamma} +\left\|\varphi_{ttt}^i\right\|_{L^1(J;H^{-1/2}(\Gamma))}\right)\right.\\
        &&+\left.\max_{t\in[0,T]}\|\varphi_{t}\|_{1,\Omega_R}\left(\max_{t\in[0,T]}\left\|\nabla\varphi_{t}^i\right\|_{-\frac{1}{2},\Gamma} +\left\|\nabla\varphi_{tt}^i\right\|_{L^1(J;H^{-1/2}(\Gamma))}\right)\right).
      \end{eqnarray}
      Combining~\eqref{eq:stable6} and~\eqref{eq:stable9}, using Young's inequality, we have:
      \begin{eqnarray}\label{eq:stable10}
       \nonumber
        &~&\left(\|\mathbf{u}_{tt}\|^2_{0,\Omega}+\|\nabla\cdot\mathbf{u}_t\|^2_{0,\Omega} +\|\varepsilon(\mathbf{u}_t)\|^2_{F(\Omega)}\right)+\left(\|\nabla\varphi_t\|^2_{0,\Omega_R} +\|\varphi_{tt}\|^2_{0,\Omega_R}\right)\leq C( E(t)+F(t))\\
        \nonumber
        &\leq& C\left(2\epsilon\max_{t\in[0,T]}\left(\|\mathbf{u}_t\|^2_{0,\Omega}+\|\nabla\cdot\mathbf{u}_t\|^2_{0,\Omega} +\|\varepsilon(\mathbf{u}_t)\|^2_{F(\Omega)}\right)+2\epsilon\max_{t\in[0,T]}\left(\|\nabla\varphi_t\|^2_{0,\Omega_R}+\|\varphi_t\|^2_{0,\Omega_R}\right)\right. \\
        \nonumber
        &~&\qquad\left.
        +\frac{1}{\epsilon}\left\|\varphi^i_t\right\|^2_{L^1(J;H^{-1/2}(\Gamma))}+\frac{1}{\epsilon}\left\|\varphi^i\right\|^2_{L^1(J;H^{-1/2}(\Gamma))}
        +\frac{1}{\epsilon}\max_{t\in[0,T]}\left\|\varphi_{tt}^i\right\|^2_{-\frac{1}{2},\Gamma} \right.\\
        &~&\qquad\left. +\frac{1}{\epsilon}\left\|\varphi_{ttt}^i\right\|^2_{L^1(J;H^{-1/2}(\Gamma))}+\frac{1}{\epsilon}\max_{t\in[0,T]}\left\|\nabla\varphi_{t}^i\right\|^2_{-\frac{1}{2},\Gamma} +\frac{1}{\epsilon}\left\|\nabla\varphi_{tt}^i\right\|^2_{L^1(J;H^{-1/2}(\Gamma))}\right).
       \end{eqnarray}
       Hence we choose $\epsilon>0$ small enough such that $2C\epsilon<1/2$. It follows from~\eqref{eq:stable10} that
       \begin{eqnarray}\label{eq:stable11}
       \nonumber
        &~&\max_{t\in[0,T]}\left(\|\mathbf{u}_{tt}\|^2_{0,\Omega}+\|\nabla\cdot\mathbf{u}_t\|^2_{0,\Omega} +\|\varepsilon(\mathbf{u}_t)\|^2_{F(\Omega)}+\|\nabla\varphi_t\|^2_{0,\Omega_R} +\|\varphi_{tt}\|^2_{0,\Omega_R}\right)\\
        \nonumber
        &\leq& C\left(\left\|\varphi^i_t\right\|^2_{L^1(J;H^{-1/2}(\Gamma))}+\left\|\nabla\varphi^i\right\|^2_{L^1(J;H^{-1/2}(\Gamma))}+\max_{t\in[0,T]}\left\|\varphi_{tt}^i\right\|^2_{-\frac{1}{2},\Gamma}\right.\\
        &&\left.+\max_{t\in[0,T]}\left\|\nabla\varphi_{t}^i\right\|^2_{-\frac{1}{2},\Gamma}
        \nonumber
        +\left\|\varphi_{ttt}^i\right\|^2_{L^1(J;H^{-1/2}(\Gamma))}+\left\|\nabla\varphi_{tt}^i\right\|^2_{L^1(J;H^{-1/2}(\Gamma))}\right).
       \end{eqnarray}
       For any $0<t\leq T$, using the Young's inequality leads to
       \begin{equation*}
         \|\varphi_t\|_{0,\Omega_R}^2=\int_{0}^{t}\partial_{\tau}\|\partial_{\tau}\varphi\|^2_{0,\Omega_R}\,d\tau \leq\epsilon T\|\varphi_{t}\|^2_{0,\Omega_R}+\frac{T}{\epsilon}\|\varphi_{tt}\|^2_{0,\Omega_R}.
       \end{equation*}
       Here we choose $\epsilon$ small enough so that $\epsilon T<1$ (e.g. $\epsilon=\frac{1}{2T}$). Hence, we have:
       \begin{equation}\label{eq:stable12}
         \|\varphi_t\|_{0,\Omega_R}^2\leq4T^2\|\varphi_{tt}\|^2_{0,\Omega_R} \leq C \|\varphi_{tt}\|^2_{0,\Omega_R}
       \end{equation}
       Similarly, we obtain
       \begin{eqnarray}
        \nonumber
         \|\nabla\varphi\|^2_{0,\Omega_R} \leq  C\|\nabla\varphi_t\|_{0,\Omega_R}^2, &&\|\mathbf{u}_{t}\|^2_{0,\Omega}\leq C\|\mathbf{u}_{tt}\|^2_{0,\Omega},\\
         \|\varepsilon(\mathbf{u})\|^2_{F(\Omega)}\leq C \|\varepsilon(\mathbf{u}_t)\|^2_{F(\Omega)},&&\|\nabla\cdot\mathbf{u}\|^2_{0,\Omega}\leq C\|\nabla\cdot\mathbf{u}_{t}\|^2_{0,\Omega}.\label{eq:stable13}
       \end{eqnarray}
       Combining \eqref{eq:stable10}-\eqref{eq:stable13}, we have
       \begin{eqnarray}\label{eq:stable14}
       \nonumber
        &~&\max_{t\in[0,T]}\left(\|\mathbf{u}_{t}\|^2_{0,\Omega}+\|\nabla\cdot\mathbf{u}\|^2_{0,\Omega} +\|\varepsilon(\mathbf{u})\|^2_{F(\Omega)}+\|\nabla\varphi\|^2_{0,\Omega_R} +\|\varphi_{t}\|^2_{0,\Omega_R}\right)\\
        \nonumber
       &\leq& C\left(\left\|\varphi^i_t\right\|^2_{L^1(J;H^{-1/2}(\Gamma))}+\left\|\nabla\varphi^i\right\|^2_{L^1(J;H^{-1/2}(\Gamma))}+\max_{t\in[0,T]}\left\|\varphi_{tt}^i\right\|^2_{-\frac{1}{2},\Gamma} \right.\\
       &&\left. +\max_{t\in[0,T]}\left\|\nabla\varphi_{t}^i\right\|^2_{-\frac{1}{2},\Gamma}
        \nonumber
      +\|\varphi_{ttt}^i\|^2_{L^1(J;H^{-1/2}(\Gamma))}+\|\nabla\varphi_{tt}^i\|^2_{L^1(J;H^{-1/2}(\Gamma))}\right),
       \end{eqnarray}
       after applying the Cauchy-Schwarz inequality we have the estimate~\eqref{eq:stability1.1} and~\eqref{eq:stability1.2}.
       \qed


\bigskip
\small
\begin{thebibliography}{99}

\bibitem{MAPMWM2006}
M. Ainsworth, P. Monk,   W. Muniz, {\it Dispersive and dissipative properties of discontinuous Galerkin finite element methods for the second-order wave equation}, J. Sci. Comput., 27(1) (2006), 5-40.

\bibitem{AAMQ2016}
 P.F. Antonietti,  B. Ayuso de Dios, I. Mazzieri,  A. Quarteroni, {\it Stability analysis of discontinuous Galerkin approximations to the elastodynamics problem}, J. Sci. Comput. 68(1) (2016),  143-170.

\bibitem{Arnold:2002}
D. Arnold, F. Brezzi, B. Cockburn, D. Marini, {\it Unified analysis of discontinuous Galerkin methods for elliptic problems}, SIAM J. Numer. Anal. 39 (2002), no. 5, 1749-1779.

\bibitem{AH2018}
 D. Appel\"{o},  T. Hagstrom, {\it  An energy-based discontinuous Galerkin discretization of the elastic wave equation in second order form}, Comput. Methods Appl. Mech. Engrg. 338 (2018), 362-391.

\bibitem{AW2018}
D. Appel\"{o}, S. Wang, {\it An energy based discontinuous galerkin method for coupled elasto-acoustic wave equations in second order form}, Submitted to International Journal for Numerical Methods in Engineering 2018, {\tt arXiv:1808.07565 [math.NA]}.

 \bibitem{BGL2018}
 G. Bao, Y. Gao, P. Li, {\it Time-domain analysis of an acoustic-elastic interaction problem},  Arch. Ration. Mech. Anal. 229(2) (2018),  835-884.

\bibitem{QCPM2014}
Q. Chen, P. Monk, {\it Discretization of the time domain CFIE for acoustic scattering problems using convolution quadrature},  SIAM J. Math. Anal., 46(5) (2014), 3107-3130.

\bibitem{ZC2009}
Z. Chen, {\it Convergence of the time-domain perfectly matched layer method for acoustic scattering problems}, Int. J. Numer. Anal. Model, 6(1) (2009), 124-146.

\bibitem{Chou:2014}
  C.-S. Chou, C.-W. Shu, Y. Xing, {\it Optimal energy conserving local discontinuous Galerkin methods for second-order wave equation in heterogeneous media}, Journal of Computational Physics, Volume 272, 2014, Pages 88-107.

\bibitem{Chung:2006}
E.T. Chung, B. Engquist, {\it Optimal discontinuous galerkin methods for wave propagation}, SIAM J. Numer. Anal., 44(5), 2006, 2131-2158.

\bibitem{Chung:2009}
E.T. Chung, B. Engquist, {\it Optimal discontinuous galerkin methods for acoustic wave equation in higher dimensions}, SIAM J. Numer. Anal., 47, 2010, 4044-4072.

\bibitem{Ciarlet:1978}
  P.G. Ciarlet, {\it The finite element method for elliptic problems}, North-Holland, Amsterdam, 1978.

\bibitem{BCCWS1989}
B. Cockburn, C.-W. Shu, {\it TVB Runge-Kutta local projection discontinuous Galerkin finite element method for conservation laws. II\@. General framework},  Math. Comput., 52(186) (1989), 411-435.

\bibitem{Cohen:2007}
  A.M. Cohen, {\it Numerical methods for Laplace transform inversion},vol.5, Numerical Methods and Algorithms. Springer, New York, 2007.

\bibitem{DLE2017}
N.F. Dudley Ward,  T. L\"{a}hivaara,  S. Eveson,  {\it A discontinuous Galerkin method for poroelastic wave propagation: the two-dimensional case}, J. Comput. Phys., 350 (2017), 690-727.

\bibitem{DGMP2016}
Y. Dudouit, L. Giraud, F. Millot, S. Pernet, {\it Interior penalty discontinuous Galerkin method for coupled elasto-acoustic media} [Research Report] RR-8986, Inria Bordeaux SudOuest. 2016.

\bibitem{OVEHA1991}
O.V. Estorff, H. Antes, {\it On FEM-BEM coupling for fluid-structure interaction analyses in the time domain}, Int. J. Numer. Meth. Engi., 31(6) (1991), 1151-1168.

\bibitem{XF2000}
 X. Feng, {\it Analysis of finite element methods and domain decomposition algorithms for a fluid-solid interaction problem}, SIAM J. Numer. Anal. 38 (4) (2000), 1312-1336.


\bibitem{XFPLYW2001}
X. Feng, P. Lee,   Y. Wei, {\it Mathematical analysis of a fluid-solid interaction problem}, Appl. Anal., 80(3-4) (2001), 409-429.

\bibitem{XFZX2004}
X. Feng, Z. Xie, {\it A priori error estimates for a coupled finite element method and mixed finite element method for a fluid-solid interaction problem}, IMA J. Numer. Anal. 24 (2004), no. 4, 671-698.

\bibitem{BFMKBIW2006}
B. Flemisch, M.Kaltenbacher,  B.I. Wohlmuth, {\it Elasto-acoustic and acoustic-acoustic coupling on non-matching grids}, Int. J. Numer. Meth. Engi., 67(13) (2006), 1791-1810.

\bibitem{GLZ2017}
Y. Gao, P. Li, B. Zhang,  {\it Analysis of transient acoustic-elastic interaction in an unbounded structure}, SIAM J. Math. Anal. 49(5) (2017), 3951-3972.

\bibitem{CGMM2017}
 C. Garc\'{i}a, G.N. Gatica, A. M{\'a}rquez, S. Meddahi, {\it A fully discrete scheme for the pressure-stress formulation of the time-domain fluid-structure interaction problem}, Calcolo 54 (4) (2017), 1419-1439.

\bibitem{CGM2017}
C. Garc\'{i}a, G.N. Gatica,   S. Meddahi, {\it Finite element semidiscretization of a pressure-stress formulation for the time-domain fluid-structure interaction problem}, IMA J. Numer. Anal. 37(4) (2017), 1772-1799.

\bibitem{GNS2018}
 H. Gimperlein, C.  \"{O}zdemir,  E.P. Stephan, {\it Time domain boundary element methods for the Neumann problem: error estimates and acoustic problems}, J. Comput. Math. 36(1) (2018),  70-89.

\bibitem{GNS2017}
H. Gimperlein,  Z. Nezhi, E.P. Stephan, {\it A priori error estimates for a time-dependent boundary element method for the acoustic wave equation in a half-space}, Math. Methods Appl. Sci. 40(2) (2017),  448-462.


\bibitem{MJGJBK1995}
M.J. Grote, J.B. Keller, {\it Exact nonreflecting boundary conditions for the time dependent wave equation}, SIAM J. Appl. Math., 55(2)(1995), 280-297.


\bibitem{MJGASDS2006}
M.J. Grote, A. Schneebeli, D. Sch\"{o}tzau, {\it Discontinuous Galerkin finite element method for the wave equation}, SIAM J. Numer. Anal. 44(6) (2006), 2408-2431.

\bibitem{TH1999}
T. Hagstrom, {\it Radiation boundary conditions for the numerical simulation of waves},  Acta numerica, 8 (1999), 47-106.

\bibitem{GCH1994}
G.C. Hsiao, {\it On the boundary-field equation methods for fluid-structure interactions}, Problems and methods in mathematical physics, Vieweg+ Teubner Verlag, 1994, 79-88.


\bibitem{GCHFJSRJW2017}
G.C. Hsiao, F.J. Sayas, R.J. Weinacht, {\it Time-dependent fluid-structure interaction}, Math. Meth. Appl. Sci., 40(2) (2017), 486-500.

\bibitem{GCHSVTFJS2017}
  G.C. Hsiao, T. S\'{a}nchez-Vizuet, F.J. Sayas, {\it Boundary and coupled boundary-finite element methods for transient wave-structure interaction}, IMA J. Numer. Anal. 37(1) (2017),  237-265.

\bibitem{CJLPAM1995}
C.J. Luke, P.A. Martin, {\it Fluid-solid interaction: acoustic scattering by a smooth elastic obstacle}, SIAM J. Appl. Math., 55(4) (1995), 904-922.

\bibitem{Kaser:2008}
M. K\"{a}ser, M. Dumbser, {\it A highly accurate discontinuous Galerkin method for complex interfaces between solids and moving fluids}, Geophysics, 73(3) (2008), T23-T35.

\bibitem{NMN1959}
N.M. Newmark, {\it A method of computation for structural dynamics}, J. Engi. Mech. Divi., 85(3) (1959), 67-94.

\bibitem{BRMFW2003}
B. Riviere, M.F. Wheeler, {\it Discontinuous finite element methods for acoustic and elastic wave problems}, Contemp. Math., 329 (2003), 271-282.

\bibitem{RD2016}
  \'{A}. Rodr\'{i}guez-Rozas,  J. Diaz, {\it  Non-conforming curved finite element schemes for time-dependent elastic-acoustic coupled problems}, J. Comput. Phys., 305 (2016), 44-62.

\bibitem{TF1975}
Tr\`eves F. {\it Basic linear partial differential equations}, Academic press, 1975.

\bibitem{LLWBWXZ2012}
L.L. Wang, B. Wang, X. Zhao, {\it Fast and accurate computation of time-domain acoustic scattering problems with exact nonreflecting boundary conditions}, SIAM J.  Appl. Math., 72(6) (2012), 1869-1898.

 \bibitem{WSBG2010}
L.C. Wilcox,  G. Stadler,  C. Burstedde, O. Ghattas, {\it A high-order discontinuous Galerkin method for wave propagation through coupled elastic-acoustic media}, J. Comput. Phys. 229(24) (2010),  9373-9396.

\bibitem{XOX2018}
J. Xie, M.Y. Ou, L. Xu, {\it A discontinuous Galerkin method for wave propagation in orthotropic poroelastic media with memory terms}, submited.


\bibitem{XCS2013}
Y. Xing, C.-S. Chou, C.-W. Shu, {\it Energy conserving local discontinuous Galerkin methods for wave propagation problems}, Inver. Prob. Image., 7(3) (2013), 967-986.


\bibitem{YX}
L. Xu, T. Yin, {\it Analysis of the Fourier series Dirichlet-to-Neumann
boundary condition of the Helmholtz equation and its application to finite element methods}, submitted.

\bibitem{XZH}
L. Xu, S. Zhang, G.C. Hsiao {\it Nonsingular kernel boundary integral and finite element coupling method}, Appl. Numer. Math., 137 (3) (2019), 80-90.

\bibitem{YDPPW2016}
  R. Ye,  M.V. de Hoop,  C.L. Petrovitch,  L.J. Pyrak-Nolte,  L.C. Wilcox, {\it A discontinuous Galerkin method with a modified penalty flux for the propagation and scattering of acousto-elastic waves},
Geophysical Journal International,  205(2) (2016), 1267-1289.


\bibitem{TYGCHLX2015}
T. Yin, G.C. Hsiao, L. Xu, {\it Boundary integral equation methods for the two-dimensional fluid-solid interaction problem}, SIAM J. Numer. Anal., 55(5) (2017), 2361-2393.


\bibitem{TYARLX}
T. Yin, A. Rathsfeld, L. Xu, {\it A BIE-based DtN-FEM for fluid-solid interaction problems}, J. Comput. Math.,  36 (1) (2018), 47-69.


\end{thebibliography}
\end{document}